\documentclass[11pt]{amsart}
\usepackage{amsmath,amssymb,latexsym,esint,cite,mathrsfs}
\usepackage{verbatim,wasysym}
\usepackage[left=2.75cm,right=2.75cm,top=2.8cm,bottom=2.8cm]{geometry}

\usepackage{microtype}
\usepackage{color,enumitem,graphicx}
\usepackage[colorlinks=true,urlcolor=blue, citecolor=red,linkcolor=blue,
linktocpage,pdfpagelabels, bookmarksnumbered,bookmarksopen]{hyperref}
\usepackage[hyperpageref]{backref}
\usepackage[english]{babel}

\newcommand{\im}{{\rm i}}
\newcommand{\eps}{{\varepsilon}}
\newtheorem{theorem}{Theorem}[section]

\newtheorem{proposition}{Proposition}[section]
\newtheorem{lemma}{Lemma}[section]

\numberwithin{equation}{section}

\title[Fractional NLS with magnetic field and critical growth]{Fractional NLS 
	equations with magnetic field, \\ critical frequency and critical growth}

\author[Zhang Binlin]{Zhang Binlin}
\author[Marco\ Squassina]{Marco\ Squassina}
\author[Zhang Xia]{Zhang Xia}

\address[Zhang Binlin]{Department of Mathematics 
		\newline\indent
	Heilongjiang Institute of Technology
		\newline\indent
	Harbin 150050, P.R. China}
\email{zhangbinlin2012@163.com}

\address[Marco Squassina]{Dipartimento di Matematica e Fisica \newline\indent
	Universit\`a Cattolica del Sacro Cuore \newline\indent
	Via dei Musei 41, I-25121 Brescia, Italy}
\email{marco.squassina@dmf.unicatt.it}

\address[Zhang Xia]{Department of Mathematics
	\newline\indent
	Harbin Institute of Technology
	\newline\indent
	Harbin 150001, P.R. China}
\email{piecesummer1984@163.com}

\begin{document}

\begin{abstract}
The paper is devoted to the study of a singularly perturbed fractional 
Schr\"{o}dinger equations involving critical frequency and critical growth in the presence of a magnetic field.
By using variational methods, we obtain the existence of  mountain pass solutions
$u_{\varepsilon}$ which tend to the trivial solution as $\varepsilon\rightarrow0$. 
Moreover, we get infinitely many solutions and sign-changing solutions for the problem in absence of
magnetic effects under some extra assumptions. 
\end{abstract}

\maketitle

\section{Introduction and main result}

In this paper,  we study  the  following Schr\"{o}dinger equations involving a critical nonlinearity
\begin{equation}
	\label{p}
	\varepsilon^{2\alpha}(-\Delta)^{\alpha}_{A_{\varepsilon}}u+V(x)u=f(x,|u|)u+K(x)|u|^{2_{\alpha}^{*}-2}u\quad\quad
	\mbox{in}\ \mathbb{R}^{N},
\end{equation}
driven by the magnetic fractional Laplacian operator $(-\Delta)^{\alpha}_{A_{\varepsilon}}$ of order $\alpha\in(0,1)$,
where $N\geq2$, $\varepsilon$ is a positive parameter, $2_{\alpha}^{*}=2N/(N-2\alpha)$ is the critical Sobolev exponent,
$V:\mathbb{R}^{N}\to\mathbb{R}$ and $A:\mathbb{R}^{N}\rightarrow\mathbb{R}^{N}$ are the electric and magnetic potentials 
respectively and $A_{\varepsilon}(x):=\varepsilon^{-1}A(x).$
If $A$ is a smooth function, the nonlocal operator
$(-\Delta)^{\alpha}_A$, which up to normalization constants can be defined on smooth functions $u$ as
$$
(-\Delta)^{\alpha}_Au(x):= 2\lim_{\varepsilon\rightarrow0}\int_{B_\varepsilon^c(x)}\frac{u(x)-e^{\im(x-y)\cdot A(\frac{x+y}{2})}u(y)}{|x-y|^{N+2\alpha}}\,dy,\ \ \ x\in\mathbb{R}^{N},
$$
has been recently introduced in \cite{DS}. The motivations for
its introduction are described in \cite{DS, SV} in more detail and rely essentially on the L\'{e}vy-Khintchine formula for
the generator of a general L\'{e}vy process. If the magnetic field $A \not \equiv 0$, it seems that the first work which considered
the existence of solutions for problem \eqref{p} in the subcritical case with $\varepsilon=1$, formally $\alpha=1$ and $K=0$ was \cite{EL}.
For more details on fractional magnetic operators we refer to \cite{TI1, TI2, TI3} for related physical background.
If the magnetic field $A\equiv0$, the above operator is consistent with the usual notion of fractional Laplacian,
which may be viewed as the infinitesimal generators of a
L\'{e}vy stable diffusion processes (see \cite{Apple2004}). This operator arises in the description of various phenomena in the applied sciences,  such as phase transitions,  materials science, conservation laws, minimal surfaces, water waves, optimization, plasma physics.
See \cite{Apple2004} and the references therein for a more detailed introduction. 
Some interesting models involving the fractional Laplacian 
have received much attention recently, such as the fractional Schr\"{o}dinger equation (see \cite{AP,Chang2013, Felmer, Laskin2000,L2}), the fractional Kirchhoff equation (see \cite{FV,PXZ}) and the fractional porous medium equation (see \cite{ JLV}).
Another driving force for the study of problem \eqref{p} arises in the study of the following time-dependent local Schr\"{o}dinger equation
\begin{equation}\label{101}
	\im\hbar \frac{\partial\psi}{\partial t}=\frac{1}{2m}\Big(\frac{\hbar}{\im} \nabla - A(x)\Big)^2\psi+W(x)\psi-g(x, |\psi|)\psi,
\end{equation}
where $\hbar$ is the Planck constant, $m$ is the mass of the particle, $A:\mathbb{R}^{N}\rightarrow\mathbb{R}^{N}$ is the magnetic potential, $W:\mathbb{R}^{N}\rightarrow\mathbb{R}^{N}$ is the electric potential,
$g$ is the nonlinear coupling and $\psi$ is the wave function representing the state of the particle.
This equation arises in Quantum Mechanics and describes the dynamics of the particle in a non-relativistic
setting, see for example \cite{RS}. Clearly, the form $\psi(x,t)=e^{-i\omega t \hbar^{-1}}u(x)$
is a standing wave solution of \eqref{101} if and only if $u$ satisfies the following stationary equation
\begin{align}\label{102}
	\Big(\frac{\varepsilon}{\im}\nabla -A(x)\Big)^{2}u+V(x)u=f(x, |u|)u.
\end{align}
where $\varepsilon=\hbar$, $V(x)=2m(W(x)-\omega)$, $f=2mg$ and
\begin{equation}
	\label{MagnOper}
	\Big(\frac{\varepsilon}{\im}\nabla-A(x)\Big)^2u=-\varepsilon^2\Delta u
	-\frac{2\varepsilon}{\im}A(x)\cdot\nabla u+|A(x)|^2u-\frac{\varepsilon}{\im}{\rm div} A(x)u.
\end{equation}
See \cite{DW} and the references cited therein for recent results in this direction (see also \cite{manmat}).
Similarly, we could derive the fractional version of (\ref{102}) as $A=0$ and $\varepsilon=1$, which is
a fundamental equation of fractional Quantum Mechanics in the study of particles on stochastic fields modeled by L\'{e}vy processes, see \cite{L2}. 
Also we refer the reader to \cite{Laskin2000} for extended physical description.

Recently, the study on fractional Schr\"{o}dinger equation has attracted much attention. On the one hand, some recent works involving the subcritical case have been obtained.\ {\em Felmer} {\em et al.} in \cite{Felmer} studied the following equations with $A=0$ and $V=1$
\begin{equation}\label{21}
	(-\Delta)^{\alpha}u+V(x)u=f(x,u).
\end{equation}
Using critical point theory, they obtained the existence of a ground state. Regularity, decay and symmetry properties of these solutions were also analyzed. In \cite{cheng}, {\em Cheng} investigated the existence of ground state for (\ref{21}) when $f(x,t)=|t|^{p-2}t$,
in which the coercivity assumption $V(x)\to+\infty$ for $|x|\to\infty$ is imposed.
In \cite{Secchi}, by using Mountain Pass arguments and a comparison method, 
{\em Secchi} considered the existence of ground state for (\ref{21})  when the potential $V$ satisfies the assumption
$\liminf_{|x|\rightarrow\infty}V(x)\geq V_\infty.$
In \cite{Ledesma}, assuming that $V^{-1}(0)$ has  nonempty interior, {\em Ledesma} obtained 
the existence of nontrivial solutions and explored the  concentration phenomenon of solutions for \eqref{21}.\
In \cite{Chen} {\em Chen} and {\em Zheng} studied the problem
\begin{eqnarray}\label{22}
	\varepsilon^{2\alpha}(-\Delta)^{\alpha}u+V(x)u=f(x,u)\quad\quad
	\mbox{in}\ \mathbb{R}^{N},
\end{eqnarray}
where $N\leq 3$, $f(x,t)=|t|^{p-2}t$ and $V(x)$ satisfies some smoothness and boundedness assumptions. By using  the Lyapunov-Schmidt reduction method, they showed that (\ref{22}) has a nontrivial solution $u_\varepsilon$ concentrating to some single point as $\varepsilon\rightarrow0$.
In \cite{Davila}, assuming that  $f(x,t)=|t|^{p-2}t$ and $V$ is a sufficiently smooth potential with $\inf_{\mathbb{R}^N}V>0$, {\em D\'{a}vila}
{\em et al.} recovered various existence results already known for the case $\alpha=1$ and showed the existence of solutions 
around $k$ nondegenerate critical points of $V$ for (\ref{22}).
In \cite{Shang},  {\em Shang} and {\em Zhang} studied the concentration phenomenon of solutions for (\ref{22}) under the assumptions $f(x,t)=K(x)|t|^{p-2}t$, $V$, $K$ are positive smooth functions and $\inf_{\mathbb{R}^N}V>0$. By a perturbative methods, they showed existence of solutions which concentrate near some critical points of the function
$$
\Gamma(x)=(V(x))^{\frac{p+2}{p}-\frac{N}{2\alpha}}(K(x))^{-\frac{2}{p+1}}.
$$
On the other hand, there are some recent papers dedicated to the study of fractional Schr\"{o}dinger equations with critical growth
under various hypotheses on the
potential function $V(x)$.
In \cite{SZ}, {\em Shang} and {\em Zhang} studied the existence for the critical fractional Schr\"{o}dinger equation
\begin{equation}\label{L1}
	\varepsilon^{2\alpha}(-\Delta)^{\alpha}u+V(x)u=\lambda f(u)+|u|^{2_{\alpha}^{*}-2}u\quad\quad
	\mbox{in}\ \mathbb{R}^{N},
\end{equation}
where  $0<\inf_{\mathbb{R}^N}V<\liminf_{|x|\rightarrow\infty}V(x)<+\infty$. Based on variational methods, they showed that problem \eqref{L1} has a nonnegative ground state solution for all sufficiently large $\lambda$ and small $\varepsilon$.
Moreover, {\em Shen} and {\em Gao} in \cite{SG} obtained the existence of nontrivial solutions for problem \eqref{L1}
under  assumptions that potential function $V$ is nonnegative and trapping, namely $\liminf_{|x|\rightarrow\infty}V(x)=+\infty$.
As for the case $\varepsilon=1$, we refer to \cite{ZZR, ZZX} for some recent results.

Motivated by the above works, especially by \cite{DW, DL},   we are interested in critical fractional Schr\"{o}dinger equations
with the magnetic field and the critical frequency case in the sense that
$\min_{\mathbb{R}^N}V=0.$ It is worth mentioning that the study of  fractional Schr\"{o}dinger equations
with the critical frequency was first investigated by {\em Byeon} and {\em Wang} \cite{BW1, BW2}.
Main difficulties arise, when dealing with this problem, because of the appearance of the magnetic field and the critical frequency,
and of the nonlocal nature of the fractional Laplacian. For this, we need to develop new techniques to overcome
difficulties induced by these new features.

\noindent
We shall assume the following conditions:

{\em\begin{itemize}
		\item[$(V_1)$] $V\in C(\mathbb{R}^N,\mathbb{R})$ and  $\min_{\mathbb{R}^N}V=0$;
		\item[$ (V_2)$] There exists  $a>0$ such that $V^{a}=\{x\in\mathbb{R}^N:V(x)<a\}$ has finite Lebesgue measure;
		\item[$(K)$] There exist  $K_0,K_1>0$ such that $K_0\leq K(x)\leq K_1$ for any $x\in\mathbb{R}^N$;
		\item[$ (f_1)$] $f\in C(\mathbb{R}^N\times\mathbb{R}^+,\mathbb{R})$ and there exists $c_0>0$ and $p\in(2,2_{\alpha}^{*})$ such that 
		$$
		|f(x,t)|\leq c_0(1+|t|^{p-2}),\quad\,\,\,\text{for any $(x,t)\in\mathbb{R}^N\times\mathbb{R}^+$;}
		$$ 
		\item[$(f_2)$] $\lim_{t\rightarrow0+}f(x,t)=0$ uniformly in $x\in\mathbb{R}^N$;
		\item[$(f_3)$] There exists $\mu>2$ such that $\mu F(x,t)\leq f(x,t)t^2$ for any $t>0$,  $F(x,t):=\int_0^t f(x,s)s\,ds$;
		\item[$(f_4)$] There exist $c_1>0$, $q\in(2,2_{\alpha}^{*})$ such that $f(x,t)\geq c_1t^{q-2}$ for any $t>0$.
	\end{itemize}}
	
	We say that
	$u\in X_\eps$ is a (weak) solution of problem \eqref{p} if for any $v\in X_\eps$,
	\begin{displaymath}
		\begin{split}
			\textup{Re}\int_{\mathbb{R}^{2N}}&\frac{\left(u(x)-e^{\im(x-y)\cdot A_{\varepsilon}(\frac{x+y}{2})}u(y)\right)
				\overline{\left(v(x)-e^{\im(x-y)\cdot A_{\varepsilon}(\frac{x+y}{2})}v(y)\right)}}{|x-y|^{N+2\alpha}}\,dxdy\\
			&+\varepsilon^{-2\alpha}\textup{Re}\int_{\mathbb{R}^{N}}V(x)u\overline{v}\,dx
			=\varepsilon^{-2\alpha}\textup{Re}\int_{\mathbb{R}^{N}}\left(f(x,|u|)u+K(x)|u|^{2_{\alpha}^{*}-2}u\right)\overline{v}\,dx.
		\end{split}
	\end{displaymath}
	\noindent where $\bar z$ denotes complex conjugate of $z\in \mathbb{C}$, Re$z$ is the real part of $z$,
	$(X_\eps,\|\cdot\|_{X_\eps})$ is a suitable subspace of the fractional space $H^{\alpha}_{A_{\varepsilon}}(\mathbb{R}^{N},\mathbb{C})$. 
	See Section 2 for more details.
\vskip4pt
\noindent
	We are now in a position to state the main result of the paper.
	
	\begin{theorem}\label{Th1}
		Assume that  $(V_1)$-$(V_2)$,
		$(f_1)$-$(f_4)$, $(K)$ hold and that $A\in C(\mathbb{R}^{N},\mathbb{R}^{N})$.\ Then there exists $\varepsilon_{0}>0$
		such that for any $\varepsilon\in(0,\varepsilon_{0})$, problem
		\textup{(1.1)} admits a nontrivial mountain pass solution  $u_{\varepsilon}\in X_\varepsilon$ such that
		$\|u_\eps\|_{X_\eps}\rightarrow0$ as $\varepsilon\rightarrow0$.
	\end{theorem}
	
	\noindent\textbf{Remark 1.1}
	
	$(i)$ unlike  solutions with concentration phenomena constructed in some earlier works without the magnetic field,
	our nontrivial solutions are closed to the trivial solution.
	
	$(ii)$ If $A=0$ and $\alpha\nearrow 1$, then Theorem \ref{Th1} reduces to a result of {\em Ding} and {\em Lin} in \cite{DL}.
	To our best knowledge, it seems that there is no result on the existence of solutions for singularly perturbed fractional Schr\"{o}dinger equations
	with an external magnetic field.
		
$(iii)$ In \cite{SV} it was proved that, in the singular limit for $\alpha\nearrow 1,$ the operator $(1-\alpha)\eps^{2\alpha}(-\Delta)^{\alpha}_{A_\eps}$
converges, in a suitable sense, to the classical local magnetic operator \eqref{MagnOper}.	Whence, up to multiplication by $1-\alpha$ the nonlocal
theory is somehow consistent with the classical one.	
		
	The paper is organized as follows.
	In Section~2, we recall some necessary definitions and properties of the functional spaces.
	In Section~3, we provide some preliminary results.
	In Section~4 we prove Theorem \ref{Th1}.
	In Section~5, we get some results for problem \textup{(1.1)} in the case $A$.

	\section{Functional Setting}\label{sec2}
	
	For the convenience of the reader, in this part we recall some definitions and basic properties of fractional magnetic Sobolev spaces $H^{\alpha}_{A_{\varepsilon}}(\mathbb{R}^{N},\mathbb{C})$.
	For a wider treatment on these spaces, we refer the reader to \cite{DS}. 
	Let $L^2(\mathbb{R}^{N},\mathbb{C})$ be the Lebesgue  space of complex-valued 
	functions with summable square, endowed with the real scalar  product
	$$
	\langle u,v\rangle_{L^2}:=\textup{Re}\int_{\mathbb{R}^{N}}u\overline{v}\,dx,
	$$
	for any $u,v\in L^2(\mathbb{R}^{N},\mathbb{C})$.
	For any $\alpha\in(0,1)$, the space $H^{\alpha}_{A_{\varepsilon}}(\mathbb{R}^{N},\mathbb{C})$ is defined by
	$$
	H^{\alpha}_{A_{\varepsilon}}(\mathbb{R}^{N},\mathbb{C})=\left\{u\in L^{2}(\mathbb{R}^{N},\mathbb{C}):[u]_{\alpha,{A_{\varepsilon}}}<\infty\right\},
	$$
	where $[u]_{\alpha,{A_{\varepsilon}}}$ denotes the so-called {\em magnetic Gagliardo semi-norm}, that is
	$$
	[u]_{\alpha,{A_{\varepsilon}}}
	:=\Big(\int_{\mathbb{R}^{2N}}\frac{|u(x)-e^{\im(x-y)\cdot {A_{\varepsilon}}(\frac{x+y}{2})}u(y)|^{2}}{|x-y|^{N+2\alpha}}\,dxdy
	\Big)^{\frac{1}{2}}
	$$
	and $H^{\alpha}_{A_{\varepsilon}}(\mathbb{R}^{N},\mathbb{C})$ is endowed with the norm
	$$
	\|u\|_{\alpha,A_{\varepsilon}}=\left([u]_{\alpha,A_{\varepsilon}}^2
	+\|u\|_{L^{2}}^2\right)^{1/2}.
	$$
	If $A=0$, then $H^{\alpha}_{A_{\varepsilon}}(\mathbb{R}^{N},\mathbb{C})$ reduces to the well-known fractional 
	space $H^{\alpha}(\mathbb{R}^{N})$. Also, $H^{\alpha}_{A_{\varepsilon}}(\mathbb{R}^{N},\mathbb{C})$ is a Hilbert space with the
	real scalar product
	$$
	\langle u,v\rangle_{\alpha,A_{\varepsilon}}:=\langle u,v\rangle_{L^2}+
	\textup{Re}\int_{\mathbb{R}^{2N}}\frac{\big(u(x)-e^{\im(x-y)\cdot A_{\varepsilon}(\frac{x+y}{2})}u(y)\big)
		\overline{\big(v(x)-e^{\im(x-y)\cdot A_{\varepsilon}(\frac{x+y}{2})}v(y)\big)}}{|x-y|^{N+2\alpha}}\,dxdy,
	$$
	for any $u,v\in H^{\alpha}_{A_{\varepsilon}}(\mathbb{R}^{N},\mathbb{C})$.
	The operator $(-\Delta)^{\alpha}_{A_{\varepsilon}}$: $H^{\alpha}_{A_{\varepsilon}}(\mathbb{R}^{N},\mathbb{C})\rightarrow H^{-\alpha}_{A_{\varepsilon}}(\mathbb{R}^{N},\mathbb{C})$ is defined by 
	\begin{equation*}
		\left\langle(-\Delta)^{\alpha}_{A_{\varepsilon}}u,v\right\rangle:=
		\textup{Re}\int_{\mathbb{R}^{2N}}\frac{\big(u(x)-e^{\im(x-y)\cdot A_{\varepsilon}(\frac{x+y}{2})}u(y)\big)
			\overline{\big(v(x)-e^{\im(x-y)\cdot A_{\varepsilon}(\frac{x+y}{2})}v(y)\big)}}{|x-y|^{N+2\alpha}}\,dxdy,
	\end{equation*}
	via duality. Furthermore, the space $D^{\alpha}_{A_{\varepsilon}}(\mathbb{R}^{N},\mathbb{C})$ is defined as 
	$$
	D^{\alpha}_{A_{\varepsilon}}(\mathbb{R}^{N},\mathbb{C}):=\big\{u\in L^{2^*_\alpha}(\mathbb{R}^{N},\mathbb{C}): [u]_{\alpha,A_{\varepsilon}}<\infty\big\}.
	$$
	and endowed with the norm $[\cdot]_{\alpha,A_{\varepsilon}}$. We recall (cf.\ \cite[Lemma 3.5]{DS}) the following embedding
	
	\begin{proposition}[Magnetic embeddings]
		\label{th2.1}  The embeddings 
		$$
		D^{\alpha}_{A_\epsilon}(\mathbb{R}^{N},\mathbb{C})\hookrightarrow L^{2_\alpha^*}(\mathbb{R}^N,\mathbb{C}), \,\,\,\quad 
		H^{\alpha}_{A_\epsilon}(\mathbb{R}^{N},\mathbb{C})\hookrightarrow L^{\nu}(\mathbb{R}^N,\mathbb{C}),
		$$ 
		is continuous for any $\nu\in [2,2_\alpha^*]$. Moreover, the embedding 
		$$
		H^{\alpha}_{A_\epsilon}(\mathbb{R}^{N},\mathbb{C})\hookrightarrow\hookrightarrow L_{{\rm loc}}^\nu(\mathbb{R}^{N},\mathbb{C})
		$$ is compact for any $\nu\in [1,2_\alpha^*)$.
	\end{proposition}
	
	\noindent
	In this paper, we will use the following subspace of $D^{\alpha}_{A_{\varepsilon}}(\mathbb{R}^{N})$ defined by
	$$
	X_\eps:=\left\{u\in D^{\alpha}_{A_{\varepsilon}}(\mathbb{R}^{N},\mathbb{C}):\int_{\mathbb{R}^{N}}V(x)|u|^2\,dx<\infty \right\}
	$$
	with the norm
	$$
	\|u\|_{X_\eps}=\left([u]_{\alpha,A_{\varepsilon}}^2+\int_{\mathbb{R}^{N}}V(x)|u|^2\,
	dx\right)^{1/2},
	$$
	where $V$ is nonnegative.
	For any $\varepsilon>0$, the norm $\|\cdot\|_{X_\eps}$ is equivalent to the following norm
	$$
	\|u\|_{\varepsilon}:=\left([u]_{\alpha,A_{\varepsilon}}^2+\varepsilon^{-2\alpha}\int_{\mathbb{R}^{N}}V(x)|u|^2\,dx\right)^{1/2},	
	$$
	which will be used from time to time.
	
	\begin{proposition}[$X_\eps$ embedding]
		\label{th2.2}  If $(V_2)$ holds, the injection
		$
		X_\eps\hookrightarrow H^{\alpha}_{A_{\varepsilon}}(\mathbb{R}^{N},\mathbb{C})
		$ 
		is continuous.
	\end{proposition}
	\begin{proof}
		Let $a>0$ be as in assumption $(V_2)$. For any $u\in X_\eps$,  we obtain
		$$
		\int_{\mathbb{R}^{N}}V(x)|u|^2\,dx=\int_{\mathbb{R}^{N}\backslash V^a}V(x)|u|^2\,dx+\int_{V^a}V(x)|u|^2\,dx.
		$$
		By the H\"{o}lder inequality,
		\begin{equation*}
			\int_{V^a}|u|^2\,dx\leq|V^a|^{1-\frac{2}{2_\alpha^*}}\left(\int_{V^a}|u|^{2_\alpha^*}\,dx\right)^{\frac{2}{2_\alpha^*}}
			\leq\frac{1}{S_{\alpha}^\eps}|V^a|^{1-\frac{2}{2_\alpha^*}}[u]_{\alpha,A_{\varepsilon}}^2,
		\end{equation*}
		where $|\cdot|$ denotes the Lebesgue measure and $S_{\alpha}^\eps$ is the best Sobolev constant of  the magnetic Sobolev embedding $D^{\alpha}_{A_{\varepsilon}}(\mathbb{R}^{N},\mathbb{C})\hookrightarrow L^{2_{\alpha}^{*}}(\mathbb{R}^{N},\mathbb{C})$, 
		\begin{eqnarray}\label{5}
			S_{\alpha}^\eps:=\inf_{u\in D^{\alpha}_{A_{\varepsilon}}(\mathbb{R}^{N})\setminus\{0\}}
			\frac{[u]_{\alpha,_{A_\varepsilon}}^2}{\|u\|_{L^{2_{\alpha}^{*}}}^{2}}.
		\end{eqnarray}
		Then, it follows from condition $(V_2)$ that
		\begin{eqnarray*}
			\begin{split}
				\|u\|_{X_\eps}^2&\geq\frac{1}{2}[u]_{\alpha,A_{\varepsilon}}^2
				+\frac{1}{2}[u]_{\alpha,A_{\varepsilon}}^2
				+\int_{\mathbb{R}^{N}\backslash V^a}V(x)|u|^2\,dx\\
				&\geq\frac{1}{2}[u]_{\alpha,A_{\varepsilon}}^2
				+\frac{1}{2}S_{\alpha}^\eps|V^a|^{\frac{2}{2_\alpha^*}-1}\int_{V^a}|u|^2\,dx
				+a\int_{\mathbb{R}^{N}\backslash V^a}|u|^2\,dx\\
				&\geq\min\left\{\frac{1}{4},\frac{1}{4}S_{\alpha}^\eps|V^a|^{\frac{2}{2_\alpha^*}-1},\frac{a}{2}\right\}\|u\|_{\alpha,A_{\varepsilon}}^2,
			\end{split}
		\end{eqnarray*}
		which implies that $X_\eps$ is continuously embedded in $H^{\alpha}_{A_{\varepsilon}}(\mathbb{R}^{N},\mathbb{C})$.
	\end{proof}

	\section{Preliminary results}\label{sec3}
	
	Throughout this section, we assume that conditions
	$(f_1)$-$(f_4)$, $(V_1)$-$(V_2)$ and $(K)$ are satisfied.
	Without loss of generality, we assume that
	$$
	V(0)=\min_{x\in\mathbb{R}^{N}}V(x)=0.
	$$
	To obtain the solution of problem (\ref{p}), we will  use  the following equivalent form
	\begin{equation}\label{23}
		(-\Delta)^{\alpha}_{A_{\varepsilon}}u+\varepsilon^{-2\alpha} V(x)u=\varepsilon^{-2\alpha} f(x,|u|)u+\varepsilon^{-2\alpha}K(x)|u|^{2_{\alpha}^{*}-2}u,
	\end{equation}
	where $\varepsilon>0$. The energy
	functional associated with (\ref{23}) on $X_\eps$ is defined as follows
	\begin{eqnarray*}
		\begin{split}
			I_\varepsilon(u):=\frac{1}{2}[u]_{\alpha,A_{\varepsilon}}^2+\frac{\varepsilon^{-2\alpha}}{2}\int_{\mathbb{R}^{N}}V(x)|u|^2\,dx
			-\varepsilon^{-2\alpha}\int_{\mathbb{R}^{N}} F(x,|u|)\,dx-\frac{\varepsilon^{-2\alpha}}{2_{\alpha}^{*}}\int_{\mathbb{R}^{N}}K(x)|u|^{2_{\alpha}^{*}}\,dx.
		\end{split}
	\end{eqnarray*}
	It is easy to check that $I_\varepsilon\in
	C^{1}(X_\eps,\,\mathbb{R})$ and that any critical point for $I_{\varepsilon}$ is a weak solution of  problem (\ref{23}).
	In the following, let $\{u_{n}\}_{n\in\mathbb N}$ be a $\mathrm{(PS)_c}$ sequence for $I_{\varepsilon}$, namely $I_{\varepsilon}(u_n)\rightarrow c$ and $I'_{\varepsilon}(u_n)\rightarrow 0$ in $X^*_\eps$, as $n\rightarrow\infty$, where $X^*_\eps$ is the dual space of $X_\eps$. 
	
	\noindent
	By standard arguments, we get that $\{u_{n}\}_{n\in \mathbb N}$ is bounded in $X_\eps$. Passing to a subsequence, still denoted by $\{u_n\}_{n\in \mathbb N}$,  we assume that $u_n\rightarrow u$ weakly in $X_\eps$, $u_n\rightarrow u$ in $L^2_{{\rm loc}}(\mathbb{R}^{N},\mathbb{C})$, $L^p_{{\rm loc}}(\mathbb{R}^{N},\mathbb{C})$ and $u_n(x)\rightarrow u(x)$ a.e.\ in $\mathbb{R}^{N}$.\  It is easy to verify that $I_{\varepsilon}'(u)=0$ and $I_{\varepsilon}(u)\geq0$.
	Due to the loss of compactness for  the critical embedding, we do not expect that  the energy functional $I_\varepsilon$  satisfies the Palais-Smale condition ((PS) condition for short) at any positive energy level, which makes the study via variational methods rather complicated.
	As in the celebrated contribution by {\em Br\'{e}zis} and {\em Nirenberg}  \cite{BN}, we show that the (PS) condition holds for energy level less than some positive constant. Then, by the Minimax Theorem, we get the existence of solutions to \eqref{23}.
	\vskip4pt
	\noindent
	First of all, we give some preliminary results to show that $I_\varepsilon$ satisfies 
	the $\mathrm{(PS)_c}$ at energy levels $c$ below some constant.
	From now on $\{u_{n}\}_{n\in \mathbb N}$ denotes the aforementioned $\mathrm{(PS)_c}$ sequence.
	
	\begin{lemma}[Vanishing]
		\label{vanish}  There is a subsequence $\{u_{n_j}\}_{j\in{\mathbb N}}$ of the $\mathrm{(PS)_c}$ sequence
		$\{u_{n}\}_{n\in{\mathbb N}}\subset X_\eps$ such that for any $\sigma>0$, there exists $r_{\sigma}>0$, which satisfies
		\begin{equation}\label{4}
			\limsup_{j\rightarrow\infty}\int_{B_j\setminus B_r}|u_{n_j}|^s\,dx\leq\sigma
		\end{equation}
		for any $r\geq r_{\sigma}$, where $s=2$ or $s=p$, and $B_r=\{x\in\mathbb{R}^{N}:|x|<r\}$.
	\end{lemma}
	\begin{proof}
		For any $r>0$, $\int_{B_r}|u_n|^s\,dx\rightarrow\int_{B_r}|u|^s\,dx$ as $n\rightarrow\infty$. Then,
		there exists $n_j\in\mathbb{N}$ with $n_{j+1}>n_{j}$ such that
		$$
		\int_{B_j}|u_{n_j}|^s\,dx-\int_{B_j}|u|^s\,dx<\frac{1}{j}.
		$$
		For any $\sigma>0$, there exists $r_{\sigma}>0$ such that for any $r\geq r_{\sigma}$,
		$$\int_{\mathbb{R}^{N}\setminus B_r}|u|^s\,dx<\sigma.$$
		If $j>r_{\sigma}$, we have
		\begin{eqnarray*}
			\begin{split}
				\int_{B_j\setminus B_r}|u_{n_j}|^s\,dx
				=&\int_{B_j}|u_{n_j}|^s\,dx-\int_{B_j}|u|^s\,dx+\int_{B_j\setminus B_r}|u|^s\,dx
				+\int_{B_r}|u|^s\,dx-\int_{B_r}|u_{n_j}|^s\,dx\\
				<&\frac{1}{j}+\sigma+\int_{B_r}|u|^s\,dx-\int_{B_r}|u_{n_j}|^s\,dx,
			\end{split}
		\end{eqnarray*}
		for any $r\geq r_{\sigma}$, which yields the desired assertion.
	\end{proof}
	
	\noindent
	Take $\varphi\in C_{0}^{\infty}(\mathbb{R}^N)$  such that $0\leq\varphi\leq1$, 
	$\varphi(x)=1$ for $|x|\leq1$ and $\varphi(x)=0$ for $|x|\geq2$. Define
	$$
	\widehat{u}_{j}(x):=\varphi_{j}(x)u(x),\quad\,\,\, \varphi_{j}(x):=\varphi\Big(\frac{2x}{j}\Big),\quad j\in\mathbb{N}.
	$$
	Then we have the following preliminary result.
	
	\begin{lemma}[Stability of truncation]
		\label{Stab}
		For any $\varepsilon>0$, $\|\widehat{u}_{j}-u\|_{\varepsilon}\rightarrow0$ as $j\rightarrow\infty$.
	\end{lemma}
	\begin{proof}
		It is readily seen that 
		\begin{eqnarray}\label{37}
			\begin{split}
				[\widehat{u}_{j}-u]_{\alpha,A_{\varepsilon}}
				\leq&2\int_{\mathbb{R}^{2N}}\frac{u^2(x)(\varphi_j(x)-\varphi_j(y))^2}{|x-y|^{N+2\alpha}}\,dxdy\\
				&+2\int_{\mathbb{R}^{2N}}\frac{(\varphi_j(y)-1)^2\big|u(x)-e^{\im(x-y)\cdot A_{\varepsilon}(\frac{x+y}{2})}u(y)\big|^2}{|x-y|^{N+2\alpha}}\,dxdy.
			\end{split}
		\end{eqnarray}
		Note that $u\in X_\eps$, $|\varphi_j(y)-1|\leq2$ and $\varphi_j(y)-1\rightarrow0$ a.e.\ as $j\to\infty$.\ Then,  the Dominated Convergence Theorem yields
		\begin{equation*}
			\int_{\mathbb{R}^{2N}}\frac{(\varphi_j(y)-1)^2\big|u(x)-e^{\im(x-y)\cdot A_{\varepsilon}(\frac{x+y}{2})}u(y)\big|^2}{|x-y|^{N+2\alpha}}\,dxdy\rightarrow0
		\end{equation*}
		as $j\rightarrow\infty$. In the following, we will prove that
		$$
		\int_{\mathbb{R}^{2N}}\frac{u^2(x)(\varphi_j(x)-\varphi_j(y))^2}{|x-y|^{N+2\alpha}}\,dxdy\rightarrow0 \ \ \ \mbox{as}\ j\rightarrow \infty.
		$$
		Note that
		\begin{displaymath}
			\begin{split}
				\mathbb{R}^{N}\times\mathbb{R}^{N}
				=&((\mathbb{R}^{N}\setminus B_j)\cup B_j)
				\times((\mathbb{R}^{N}\setminus B_j)\cup B_j)\\
				=&((\mathbb{R}^{N}\setminus B_j)\times(\mathbb{R}^{N}\setminus B_j))\cup(B_j\times\mathbb{R}^{N})\cup((\mathbb{R}^{N}\setminus B_j)\times B_j).
			\end{split}
		\end{displaymath}
		(i) $(x,y)\in(\mathbb{R}^{N}\setminus B_j)\times(\mathbb{R}^{N}\setminus B_j)$, we have $\varphi_{j}(x)=\varphi_{j}(y)=0$. \newline
		(ii) $(x,y)\in B_j\times\mathbb{R}^{N}$. One has
		\begin{displaymath}
			\begin{split}
				&\int_{B_j}dx\int_{\{y\in\mathbb{R}^{N}:|x-y|\leq\frac{1}{2}j\}}\frac{u^{2}(x)|\varphi_{j}(x)-\varphi_{j}(y)|^{2}}{|x-y|^{N+2\alpha}}dy\\
				&=\int_{B_j}dx\int_{\{y\in\mathbb{R}^{N}:|x-y|\leq\frac{1}{2}j\}}\frac{u^{2}(x)|\nabla\varphi(\xi)|^{2}|\frac{2(x-y)}{j}|^{2}}{|x-y|^{N+2\alpha}}dy\\
				&\leq Cj^{-2}
				\int_{B_j}dx\int_{\{y\in\mathbb{R}^{N}:|x-y|\leq\frac{1}{2}j\}}\frac{u^{2}(x)}{|x-y|^{N+2\alpha-2}}dy\\
				&=Cj^{-2\alpha}
				\int_{B_j}u^{2}(x)dx,
			\end{split}
		\end{displaymath}
		where $\xi=\frac{2y}{j}+\tau\frac{2(x-y)}{j}$, $\tau\in(0,1)$ and
		\begin{displaymath}
			\begin{split}
				&\int_{B_j}dx\int_{\{y\in\mathbb{R}^{N}:|x-y|>\frac{1}{2}j\}}\frac{u^{2}(x)|\varphi_{j}(x)-\varphi_{j}(y)|^{2}}{|x-y|^{N+2\alpha}}dy\\
				&\leq C
				\int_{B_j}dx\int_{\{y\in\mathbb{R}^{N}:|x-y|>\frac{1}{2}j\}}\frac{u^{2}(x)}{|x-y|^{N+2\alpha}}dy\\
				&=Cj^{-2\alpha}
				\int_{B_j}u^{2}(x)dx.
			\end{split}
		\end{displaymath}
		(iii) $(x,y)\in(\mathbb{R}^{N}\setminus B_j)\times B_j$.
		If $|x-y|\leq\frac{1}{2}j$, then $|x|\leq|x-y|+|y|\leq\frac{3}{2}j.$ Furthermore,
		\begin{displaymath}
			\begin{split}
				&\int_{\mathbb{R}^{N}\setminus B_j}dx\int_{\{y\in B_j:|x-y|\leq\frac{1}{2}j\}}\frac{u^{2}(x)|\varphi_{j}(x)-\varphi_{j}(y)|^{2}}{|x-y|^{N+2\alpha}}dy\\
				&\leq Cj^{-2}
				\int_{B_{\frac{3}{2}j}}dx\int_{\{y\in B_j:|x-y|\leq\frac{1}{2}j\}}\frac{u^{2}(x)}{|x-y|^{N+2\alpha-2}}dy\\
				&\leq Cj^{-2\alpha}
				\int_{B_{\frac{3}{2}j}}u^{2}(x)dx.
			\end{split}
		\end{displaymath}
		Notice that, for any $k>4$, there holds
		$$
		\mathbb{R}^{N}\setminus B_j
		\subset B_{\frac{k}{2}j}\cup (\mathbb{R}^{N}\setminus B_{\frac{k}{2}j}).
		$$
		If $|x-y|>\frac{1}{2}j$, then we obtain
		\begin{displaymath}
			\begin{split}
				&\int_{B_{\frac{k}{2}j}}dx\int_{\{y\in B_j:|x-y|>\frac{1}{2}j\}}\frac{u^{2}(x)|\varphi_{j}(x)-\varphi_{j}(y)|^{2}}{|x-y|^{N+2\alpha}}dy\\
				&\leq C
				\int_{B_{\frac{k}{2}j}}dx\int_{\{y\in B_j:|x-y|>\frac{1}{2}j\}}\frac{u^{2}(x)}{|x-y|^{N+2\alpha}}dy\\
				&\leq Cj^{-2\alpha}
				\int_{B_{\frac{k}{2}j}}u^{2}(x)dx.
			\end{split}
		\end{displaymath}
		If $(x,y)\in(\mathbb{R}^{N}\setminus B_{\frac{k}{2}j})\times B_j$, then
		$|x-y|\geq|x|-|y|\geq\frac{|x|}{2}+\frac{k}{4}j-j>\frac{|x|}{2}.$
		H\"{o}lder inequality yields
		\begin{displaymath}
			\begin{split}
				&\int_{\mathbb{R}^{N}\setminus B_{\frac{k}{2}j}}dx\int_{\{y\in B_j:|x-y|>\frac{1}{2}j\}}\frac{u^{2}(x)|\varphi_{j}(x)-\varphi_{j}(y)|^{2}}{|x-y|^{N+2\alpha}}dy\\
				&\leq C
				\int_{\mathbb{R}^{N}\setminus B_{\frac{k}{2}j}}dx\int_{\{y\in B_j:|x-y|>\frac{1}{2}j\}}\frac{u^{2}(x)}{|x|^{N+2\alpha}}dy\\
				&\leq Cj^{N}
				\int_{\mathbb{R}^{N}\setminus B_{\frac{k}{2}j}}\frac{u^{2}(x)}{|x|^{N+2\alpha}}dx\\
				&\leq Ck^{-N}\left(\int_{\mathbb{R}^{N}\setminus B_{\frac{k}{2}j}}|u(x)|^{2_{\alpha}^{*}}\,dx\right)^{\frac{2}{2_{\alpha}^{*}}}.
			\end{split}
		\end{displaymath}
		By combining (i), (ii) and (iii), we get
		\begin{displaymath}
			\begin{split}
				&\int_{\mathbb{R}^{2N}}\frac{u^{2}(x)|\varphi_{j}(x)-\varphi_{j}(y)|^{2}}{|x-y|^{N+2\alpha}}\,dxdy\\
				&=\left(\int_{B_j\times\mathbb{R}^{N}} +\int_{(\mathbb{R}^{N}\setminus B_j)\times B_j}\right) \frac{u^{2}(x)|\varphi_{j}(x)-\varphi_{j}(y)|^{2}}{|x-y|^{N+2\alpha}}\,dxdy\\
				&\leq Cj^{-2\alpha}\int_{B_{\frac{k}{2}j}}u^{2}(x)\,dx+Ck^{-N}\left(\int_{\mathbb{R}^{N}\setminus B_{\frac{k}{2}j}}|u(x)|^{2_{\alpha}^{*}}\,dx\right)^{2/2_{\alpha}^{*}}\\
				&\leq Cj^{-2\alpha}+Ck^{-N}.
			\end{split}
		\end{displaymath}
		Therefore, we have
		\begin{eqnarray}\label{39}
			\begin{split}
				&\limsup_{j\rightarrow\infty}
				\int_{\mathbb{R}^{2N}}\frac{u^{2}(x)(\varphi_{j}(x)-\varphi_{j}(y))^{2}}{|x-y|^{N+2\alpha}}\,dxdy\\
				&=\lim_{k\rightarrow\infty}\limsup_{j\rightarrow\infty}
				\int_{\mathbb{R}^{2N}}\frac{u^{2}(x)(\varphi_{j}(x)-\varphi_{j}(y))^{2}}{|x-y|^{N+2\alpha}}\,dxdy=0.
			\end{split}
		\end{eqnarray}
		It follows from \eqref{37} and \eqref{39} that
		$$
		[\widehat{u}_{j}-u]_{\alpha,A_{\varepsilon}}\rightarrow 0, \ \ \mbox{as}\  j\rightarrow\infty.
		$$
		Note that, as $j\rightarrow\infty$,
		$$
		\int_{\mathbb{R}^{N}}V(x)|\widehat{u}_{j}(x)-u(x)|^2\,dx=\int_{\mathbb{R}^{N}}V(x)(\varphi_j(x)-1)^2u^2(x)\,dx\rightarrow0,
		$$
		we deduce from the Dominated Convergence Theorem that
		$$
		\int_{\mathbb{R}^{N}}V(x)|\widehat{u}_{j}(x)-u(x)|^2\,dx\rightarrow 0,
		$$ 
		as $j\rightarrow\infty$. Thus $\|\widehat{u}_{j}-u\|_{\varepsilon}\rightarrow0$ as $j\rightarrow\infty$.
	\end{proof}
	
	\begin{lemma}[Shifted Palais-Smale]
		\label{Shift}
Let $\{u_{n_j}\}_{j\in {\mathbb N}}\subset X_\eps$ be the sequence introduced in Lemma~\ref{vanish}.
		Moreover, for any $j\in\mathbb{N}$, let us denote 
		$$
		u^{1}_{n_j}:=u_{n_j}-\widehat{u}_j,\quad j\geq 1.
		$$
		Then, 
		$I_{\varepsilon}(u^{1}_{n_j})\rightarrow c-I_{\varepsilon}(u)$ and
		$I_{\varepsilon}'(u^{1}_{n_j})\rightarrow 0$ in $X^*_\eps$ as $j\rightarrow\infty$.
	\end{lemma}
	\begin{proof} 
		 Notice that, it holds
		\begin{eqnarray*}
			\begin{split}
				&I_{\varepsilon}(u^{1}_{n_j})-I_{\varepsilon}(u_{n_j})+I_{\varepsilon}(\widehat{u}_j)\\
				=&\int_{\mathbb{R}^{2N}}\frac{|\widehat{u}_j(x)-e^{\im(x-y)\cdot A_{\varepsilon}(\frac{x+y}{2})}\widehat{u}_j(y)|^2}{|x-y|^{N+2\alpha}}\,dxdy\\
				&-\textup{Re}
				\int_{\mathbb{R}^{2N}}\frac{\left(u_{n_j}(x)-e^{\im(x-y)\cdot A_{\varepsilon}(\frac{x+y}{2})}u_{n_j}(y)\right)
					\overline{\left(\widehat{u}_j(x)-e^{\im(x-y)\cdot A_{\varepsilon}(\frac{x+y}{2})}\widehat{u}_j(y)\right)}}{|x-y|^{N+2\alpha}}\,dxdy\\
				&+\varepsilon^{-2\alpha}\int_{\mathbb{R}^{N}}V(x)|\widehat{u}_j|^2\,dx
				-\varepsilon^{-2\alpha}\textup{Re}\int_{\mathbb{R}^{N}}V(x)u_{n_j}\overline{\widehat{u}_j}\,dx\\
				&+\varepsilon^{-2\alpha}\int_{\mathbb{R}^{N}}(F(x,|u_{n_j}|)-F(x,|u_{n_j}-\widehat{u}_j|)-F(x,|\widehat{u}_j|))\,dx\\
				&+\frac{\varepsilon^{-2\alpha}}{2_{\alpha}^{*}}\int_{\mathbb{R}^{N}}K(x)\left(|u_{n_j}|^{2_{\alpha}^{*}}
				-|u_{n_j}-\widehat{u}_j|^{2_{\alpha}^{*}}-|\widehat{u}_j|^{2_{\alpha}^{*}}\right)\,dx.
			\end{split}
		\end{eqnarray*}
		As $u_{n_{j}}\rightarrow u$ weakly in $X_\eps$ and $\widehat{u}_{j}\rightarrow u$ strongly in $X_\eps$, we could derive that
		\begin{align*}
			\int_{\mathbb{R}^{2N}}&\frac{|\widehat{u}_j(x)-e^{\im(x-y)\cdot A_{\varepsilon}(\frac{x+y}{2})}\widehat{u}_j(y)|^2}{|x-y|^{N+2\alpha}}\,dxdy\\
			&-\textup{Re}
			\int_{\mathbb{R}^{2N}}\frac{\left(u_{n_j}(x)-e^{\im(x-y)\cdot A_{\varepsilon}(\frac{x+y}{2})}u_{n_j}(y)\right)
				\overline{\left(\widehat{u}_j(x)-e^{\im(x-y)\cdot A_{\varepsilon}(\frac{x+y}{2})}\widehat{u}_j(y)\right)}}{|x-y|^{N+2\alpha}}\,dxdy\rightarrow 0
		\end{align*}
		and, as $j\to\infty$,
		$$
		\int_{\mathbb{R}^{N}}V(x)|\widehat{u}_j|^2\,dx
		-\textup{Re}\int_{\mathbb{R}^{N}}V(x)u_{n_j}\overline{\widehat{u}_j}\,dx\rightarrow 0.
		$$
		Arguing as for the proof of  the Br\'{e}zis-Lieb Lemma and recalling that $\widehat u_j\to u$
		strongly in $X_\eps$ as $j\rightarrow\infty$, it is easy to prove that
		\begin{align*}
		&\int_{\mathbb{R}^{N}}(F(x,u_{n_j})-F(x,u_{n_j}-\widehat{u}_j)-F(x,\widehat{u}_j))\,dx\rightarrow 0, \\ 
		&\int_{\mathbb{R}^{N}}K(x)(|u_{n_j}|^{2_{\alpha}^{*}}
		-|u_{n_j}-\widehat{u}_j|^{2_{\alpha}^{*}}-|\widehat{u}_j|^{2_{\alpha}^{*}})\,dx\rightarrow 0.
		\end{align*}
		Thus, $I_{\varepsilon}(u^{1}_{n_j})\rightarrow c-I_{\varepsilon}(u)$, as $j\rightarrow\infty$. 
		Taking now $\phi\in X_\eps$ with $\|\phi\|_{\varepsilon}\leq1$, we obtain
		\begin{eqnarray*}
			\begin{split}
				&\langle I_{\varepsilon}'(u^{1}_{n_j})-I_{\varepsilon}'(u_{n_j})+I_{\varepsilon}'(\widehat{u}_j),\phi\rangle\\
				=&\varepsilon^{-2\alpha}\textup{Re}\int_{\mathbb{R}^{N}}(f(x,|u_{n_j}|)u_{n_j}-f(x,|u_{n_j}-\widehat{u}_j|)(u_{n_j}-\widehat{u}_j)-f(x,|\widehat{u}_j|)\widehat{u}_j)\overline{\phi}\,dx\\
				&+\varepsilon^{-2\alpha}\textup{Re}\int_{\mathbb{R}^{N}}K(x)(|u_{n_j}|^{2_{\alpha}^{*}-2}u_{n_j}
				-|u_{n_j}-\widehat{u}_j|^{2_{\alpha}^{*}-2}(u_{n_j}-\widehat{u}_j)-|\widehat{u}_j|^{2_{\alpha}^{*}-2}\widehat{u}_j)\overline{\phi}\,dx.
			\end{split}
		\end{eqnarray*}
		It follows, again by a standard argument, that
		$$
		\left|\int_{\mathbb{R}^{N}}K(x)(|u_{n_j}|^{2_{\alpha}^{*}-2}u_{n_j}
		-|u_{n_j}-\widehat{u}_j|^{2_{\alpha}^{*}-2}(u_{n_j}-\widehat{u}_j)-|\widehat{u}_j|^{2_{\alpha}^{*}-2}\widehat{u}_j)\overline{\phi}\,dx\right|\rightarrow 0
		$$
		uniformly in $\phi\in X_\eps$ with $\|\phi\|_{\varepsilon}\leq1$, as $j\rightarrow\infty$.
		Meanwhile, we have
		\begin{eqnarray*}
			\begin{split}
				&\left|\int_{\mathbb{R}^{N}}(f(x,|u_{n_j}|)u_{n_j}-f(x,|u_{n_j}-\widehat{u}_j|)(u_{n_j}-\widehat{u}_j)-f(x,|\widehat{u}_j|)\widehat{u}_j)\overline{\phi}\,dx\right|\\
				\leq&\int_{B_r}\left|f(x,|u_{n_j}|)u_{n_j}-f(x,|u_{n_j}-\widehat{u}_j|)(u_{n_j}-\widehat{u}_j)-f(x,|\widehat{u}_j|)\widehat{u}_j\right|\cdot|\phi|\,dx\\
				&+\int_{\mathbb{R}^{N}\setminus B_r}\left|f(x,|u_{n_j}|)u_{n_j}-f(x,|u_{n_j}-\widehat{u}_j|)(u_{n_j}-\widehat{u}_j)-f(x,|\widehat{u}_j|)\widehat{u}_j\right|\cdot|\phi|\,dx
			\end{split}
		\end{eqnarray*}
		for any  $r\geq r_\sigma$, where $r_\sigma>0$ is as in Lemma~\ref{vanish}. Since
		$\widehat{u}_{j}\rightarrow u$ and $u_{n_j}\rightarrow u$ in $L^p(B_r,\mathbb{C})$, we get
		\begin{equation}\label{6}
			\int_{B_r}|f(x,|u_{n_j}|)u_{n_j}-f(x,|u_{n_j}-\widehat{u}_j|)(u_{n_j}-\widehat{u}_j)-f(x,|\widehat{u}_j|)\widehat{u}_j|\cdot|\phi|\,dx\rightarrow 0
		\end{equation}
		uniformly in $\phi\in X_\eps$ with $\|\phi\|_{\varepsilon}\leq1$. By $(f_1)$ and $(f_2)$, for any $t>0$ we obtain
		$$
		|f(x,t)t|\leq C(|t|+|t|^{p-1}),
		$$
		which implies (we recall that $\hat u_j=0$ on ${\mathbb R}^N\setminus B_j$ for any $j\geq 1$)
		\begin{eqnarray*}
			\begin{split}
				&\int_{\mathbb{R}^{N}\setminus B_r}|f(x,|u_{n_j}|)u_{n_j}-f(x,|u_{n_j}-\widehat{u}_j|)(u_{n_j}-\widehat{u}_j)-f(x,|\widehat{u}_j|)\widehat{u}_j|\cdot|\phi|\,dx\\
				=&\int_{B_j\setminus B_r}|f(x,|u_{n_j}|)u_{n_j}-f(x,|u_{n_j}-\widehat{u}_j|)(u_{n_j}-\widehat{u}_j)-f(x,|\widehat{u}_j|)\widehat{u}_j|\cdot|\phi|\,dx\\
				\leq&C\int_{B_j\setminus B_r}(|u_{n_{j}}|+|\widehat{u}_{j}|+|u_{n_{j}}|^{p-1}+|\widehat{u}_{j}|^{p-1})\cdot|\phi|\,dx.
			\end{split}
		\end{eqnarray*}
		For any $\sigma>0$, by inequality \eqref{4}, the H\"{o}lder inequality and Proposition~\ref{th2.2}, we have
		\begin{align*}
				&\limsup_{j\rightarrow\infty}\int_{B_j\setminus B_r}(|u_{n_{j}}|+|u_{n_{j}}|^{p-1})\cdot|\phi|\,dx\\
				\leq&\limsup_{j\rightarrow\infty}\Big(\int_{B_j\setminus B_r}|u_{n_{j}}|^2\,dx\Big)^{\frac{1}{2}}\Big(\int_{B_j\setminus B_r}|\phi|^2\,dx\Big)^{\frac{1}{2}}\\
				+&\limsup_{j\rightarrow\infty}\Big(\int_{B_j\setminus B_r}|u_{n_{j}}|^p\,dx\Big)^{\frac{p-1}{p}}
				\Big(\int_{B_j\setminus B_r}|\phi|^p\,dx\Big)^{\frac{1}{p}}
				\leq C\big(\sigma^{\frac{1}{2}}+\sigma^{\frac{p-1}{p}}\big).
		\end{align*}
		Since $\widehat{u}_{j}\rightarrow u$ in $X_\eps$ as $j\rightarrow \infty$, Proposition \ref{th2.2}  yields that $\widehat{u}_{j}\rightarrow u$ in $L^2(\mathbb{R}^{N},\mathbb{C})$ and $L^p(\mathbb{R}^{N},\mathbb{C})$. Then, by the H\"{o}lder inequality, for any $r\geq r_\sigma$ 
		(up to enlarging $r_\sigma$) we obtain
		\begin{equation}
		\label{8}
				\limsup_{j\rightarrow\infty}\int_{B_j\setminus B_r}(|\widehat{u}_{j}|+|\widehat{u}_{j}|^{p-1})\cdot|\phi|\,dx
				=\int_{\mathbb{R}^{N}\setminus B_r}(|u|+|u|^{p-1})\cdot|\phi|\,dx 
		 		\leq C\left(\sigma^{\frac{1}{2}}+\sigma^{\frac{p-1}{p}}\right).    
		\end{equation}
		From \eqref{6}-\eqref{8}, we have
		\begin{eqnarray*}
			\begin{split}
				\limsup_{j\rightarrow\infty}\int_{\mathbb{R}^{N}}|f(x,|u_{n_j}|)u_{n_j}-f(x,|u_{n_j}-\widehat{u}_j|)(u_{n_j}-\widehat{u}_j)-f(x,|\widehat{u}_j|)\widehat{u}_j|\cdot|\phi|\,dx\leq C\left(\sigma^{\frac{1}{2}}+\sigma^{\frac{p-1}{p}}\right)
			\end{split}
		\end{eqnarray*}
		uniformly in $\phi\in X_\eps$ with $\|\phi\|_{\varepsilon}\leq1$. Letting $\sigma\rightarrow0$ yields,
		$$\limsup_{j\rightarrow\infty}\int_{\mathbb{R}^{N}}|f(x,|u_{n_j}|)u_{n_j}-f(x,|u_{n_j}-
		\widehat{u}_j|)(u_{n_j}-\widehat{u}_j)-f(x,|\widehat{u}_j|)\widehat{u}_j|\cdot|\phi|\,dx=0.$$
		As $I'_{\varepsilon}(u_{n_j})\rightarrow0$ and $I'_{\varepsilon}(\widehat{u}_j)\rightarrow I'_{\varepsilon}(u)=0$, we get that $I_{\varepsilon}'(u^{1}_{n_j})\rightarrow0$, as $j\rightarrow\infty$.
	\end{proof}
	
	\noindent
	In what follows, we will show that for any $\varepsilon>0$,  $I_\varepsilon$ satisfies 
	$\mathrm{(PS)_c}$ condition for energy level $c$ below some positive constant depending on $\varepsilon$.
	
	\begin{lemma}[Palais-Smale]
		\label{Th3.1} 
		Let $K_0,K_1>0$ and $\mu>2$ be as in conditions $(K)$ and $(f_3)$ and let us denote by $c_2>0$
		a suitable constant depending upon $f$. Then, for any $\varepsilon>0$, if 
		$$
		-\infty<c<C_0(\eps)\varepsilon^{N-2\alpha}, \qquad 
		C_{0}(\eps):=\left(\frac{S_{\alpha}^\eps}{c_2+K_1}\right)^{\frac{2^{*}_{\alpha}}{2^{*}_{\alpha}-2}}\frac{(2^{*}_{\alpha}-\mu)K_0}{\mu 2^{*}_{\alpha}}.
		$$
		then
		$u_{n_j}\rightarrow u$ in $X_\eps$ as $j\rightarrow\infty$.
	\end{lemma}
	\begin{proof}By the definition of $\{u_{n_j}^1\}_{j\in {\mathbb N}}$ and Lemma \ref{Stab}, it suffices to 
		have $u^1_{n_j}\rightarrow0$ in $X_\eps$ as $j\rightarrow\infty$. By means of conditions $(f_3)$ and $(K)$, we have
		\begin{eqnarray*}\label{9}
			\begin{split}
				&I_{\varepsilon}(u^{1}_{n_j})-\frac{1}{\mu}\langle I'_{\varepsilon}(u^{1}_{n_j}),u^{1}_{n_j}\rangle\\
				=&\left(\frac{1}{2}
				-\frac{1}{\mu}\right)[u^{1}_{n_j}]^2_{\alpha,A_{\varepsilon}}
				+\left(\frac{1}{2}
				-\frac{1}{\mu}\right)\varepsilon^{-2\alpha}\int_{\mathbb{R}^{N}}V(x)|u^{1}_{n_j}|^2\,dx\\
				&+\varepsilon^{-2\alpha}\int_{\mathbb{R}^{N}}\left(\frac{1}{\mu}f(x,|u^{1}_{n_j}|)|u^{1}_{n_j}|^2-F(x,|u^{1}_{n_j}|)\right)\,dx
				+\varepsilon^{-2\alpha}\left(\frac{1}{\mu}
				-\frac{1}{2^{*}_{\alpha}}\right)\int_{\mathbb{R}^{N}}K(x)|u^{1}_{n_j}|^{2^{*}_{\alpha}}\,dx\\
				\geq&\varepsilon^{-2\alpha}\left(\frac{1}{\mu}
				-\frac{1}{2^{*}_{\alpha}}\right)K_{0}\int_{\mathbb{R}^{N}}|u^{1}_{n_j}|^{2^{*}_{\alpha}}\,dx.
			\end{split}
		\end{eqnarray*}
		Then, from Lemma~\ref{Shift}, we get
		\begin{eqnarray}\label{15}
			\begin{split}
				\limsup_{j\rightarrow\infty}\int_{\mathbb{R}^{N}}|u^{1}_{n_j}|^{2^{*}_{\alpha}}\,dx\leq\frac{\mu 2^{*}_{\alpha}\varepsilon^{2\alpha}}{K_0(2^{*}_{\alpha}-\mu)}(c-I_{\varepsilon}(u)).
			\end{split}
		\end{eqnarray}
		Suppose that  $u^{1}_{n_j}\not\to 0$ in $L^{2^{*}_{\alpha}}(\mathbb{R}^{N},\mathbb{C})$. Then, we have
		\begin{equation}\label{14}
			\liminf_{j\rightarrow\infty}\|u^{1}_{n_j}\|_{L^{2^{*}_{\alpha}}}>0.
		\end{equation}
		Noting that $\langle I'_{\varepsilon}(u^{1}_{n_j}),u^{1}_{n_j}\rangle\rightarrow0$ as $j\rightarrow \infty$, we have
		\begin{eqnarray}\label{17}
			\begin{split}
				&[u^{1}_{n_j}]^2_{\alpha,A_{\varepsilon}}+\varepsilon^{-2\alpha}\int_{\mathbb{R}^{N}}V(x)|u^{1}_{n_j}|^{2}\,dx\\
				=&\varepsilon^{-2\alpha}\int_{\mathbb{R}^{N}}f(x,|u^{1}_{n_j}|)|u^{1}_{n_j}|^2\,dx
				+\varepsilon^{-2\alpha}\int_{\mathbb{R}^{N}}K(x)|u^{1}_{n_j}|^{2^{*}_{\alpha}}\,dx+o_j(1).
			\end{split}
		\end{eqnarray}
		It follows from (\ref{5}) that
		\begin{eqnarray*}
			\begin{split}
				&S_{\alpha}^\eps\left(\int_{\mathbb{R}^{N}}|u^{1}_{n_j}|^{2^{*}_{\alpha}}\,dx\right)^{\frac{2}{2^{*}_{\alpha}}}
				\leq[u^{1}_{n_j}]^2_{\alpha,A_{\varepsilon}}\\
				=&\varepsilon^{-2\alpha}\int_{\mathbb{R}^{N}}f(x,|u^{1}_{n_j}|)|u^{1}_{n_j}|^2\,dx
				+\varepsilon^{-2\alpha}\int_{\mathbb{R}^{N}}K(x)|u^{1}_{n_j}|^{2^{*}_{\alpha}}\,dx
				-\varepsilon^{-2\alpha}\int_{\mathbb{R}^{N}}V(x)|u^{1}_{n_j}|^{2}\,dx+o_j(1).
			\end{split}
		\end{eqnarray*}
		By $(f_1)$ and $(f_2)$, for any $\lambda>0$, there exists $C(\lambda)>0$ such that
		\begin{equation}\label{16}
			|f(x,t)|\leq\lambda+C(\lambda)|t|^{2^{*}_{\alpha}-2}.
		\end{equation}
		Thus
		\begin{eqnarray}\label{10}
			\begin{split}
				S_{\alpha}^\eps\left(\int_{\mathbb{R}^{N}}|u^{1}_{n_j}|^{2^{*}_{\alpha}}\,dx\right)^{\frac{2}{2^{*}_{\alpha}}}
				\leq&\lambda\varepsilon^{-2\alpha}\int_{\mathbb{R}^{N}}|u^{1}_{n_j}|^{2}\,dx
				+C(\lambda)\varepsilon^{-2\alpha}\int_{\mathbb{R}^{N}}|u^{1}_{n_j}|^{2^{*}_{\alpha}}\,dx\\
				&+\varepsilon^{-2\alpha}K_1\int_{\mathbb{R}^{N}}|u^{1}_{n_j}|^{2^{*}_{\alpha}}\,dx
				-\varepsilon^{-2\alpha}\int_{\mathbb{R}^{N}}a|u^{1}_{n_j}|^{2}\,dx\\
				&+\varepsilon^{-2\alpha}\int_{\mathbb{R}^{N}}a|u^{1}_{n_j}|^{2}\,dx
				-\varepsilon^{-2\alpha}\int_{\mathbb{R}^{N}}V(x)|u^{1}_{n_j}|^{2}\,dx+o_j(1).
			\end{split}
		\end{eqnarray}
		Since $V^{a}$ has finite Lebesgue measure, we obtain $|V^{a}\setminus B_R|\to 0$ for $R\to\infty$.
		Then, for any $\eta>0$, there exists $R_0>0$ such that $|V^{a}\setminus B_R|<\eta$ for any $R\geq R_0$.
		We have
		\begin{align*}
				\int_{\mathbb{R}^{N}}(a-V(x))|u^{1}_{n_j}|^{2}\,dx
				\leq&\int_{V^a}(a-V(x))|u^{1}_{n_j}|^{2}\,dx\\
				=&\int_{V^a\setminus B_{R_0}}(a-V(x))|u^{1}_{n_j}|^{2}\,dx
				+\int_{V^a \cap B_{R_0}}(a-V(x))|u^{1}_{n_j}|^{2}\,dx.
			\end{align*}
		Now the H\"{o}lder inequality gives
		\begin{align*}
				\int_{V^a\setminus B_{R_0}}(a-V(x))|u^{1}_{n_j}|^{2}\,dx
				\leq&\int_{V^a\setminus B_{R_0}}a|u^{1}_{n_j}|^{2}\,dx\\
				\leq&a\|u^{1}_{n_j}\|^2_{L^{2^*_{\alpha}}}|V^a\setminus B_{R_0}|^{1-\frac{2}{2^*_{\alpha}}}
				\leq C\eta^{1-\frac{2}{2^*_{\alpha}}}.
			\end{align*}
		As $u^{1}_{n_j}\rightarrow0$ weakly in $X$, $u^{1}_{n_j}\rightarrow0$ in $L^2(B_{R_0},\mathbb{C})$, as $j\rightarrow\infty$. Then, for the above $\eta>0$,
		there exists $j_{0}\in\mathbb{N}$ such that for any $j\geq j_0$,
		\begin{equation}\label{13}
			\int_{V^a \cap B_{R_0}}(a-V(x))|u^{1}_{n_j}|^{2}\,dx\leq a\int_{B_{R_0}}|u^{1}_{n_j}|^{2}\,dx\leq a\eta.
		\end{equation}
		Let $\lambda=a/2$. In terms of \eqref{10}-\eqref{13}, there exists $c_2>0$ depending on $f$ such that
		\begin{eqnarray*}
			\begin{split}
				S_{\alpha}^\eps\left(\int_{\mathbb{R}^{N}}|u^{1}_{n_j}|^{2^{*}_{\alpha}}\,dx\right)^{\frac{2}{2^{*}_{\alpha}}}
				\leq(c_2+K_1)\varepsilon^{-2\alpha}\int_{\mathbb{R}^{N}}|u^{1}_{n_j}|^{2^{*}_{\alpha}}\,dx+
				C\varepsilon^{-2\alpha}\eta^{1-\frac{2}{2^{*}_{\alpha}}}+\varepsilon^{-2\alpha} a\eta+o_j(1).
			\end{split}
		\end{eqnarray*}
		Letting $\eta\rightarrow0$,
		we have
		$$
		S_{\alpha}^\eps\left(\int_{\mathbb{R}^{N}}|u^{1}_{n_j}|^{2^{*}_{\alpha}}\,dx\right)^{\frac{2}{2^{*}_{\alpha}}}
		\leq(c_2+K_1)\varepsilon^{-2\alpha}\int_{\mathbb{R}^{N}}|u^{1}_{n_j}|^{2^{*}_{\alpha}}\,dx+o_j(1).
		$$
		From (\ref{15}) and (\ref{14}), we get
		$$
		S_{\alpha}^\eps\leq(c_2+K_1)\varepsilon^{-2\alpha}\left(\frac{\mu 2^{*}_{\alpha}\varepsilon^{2\alpha}}{K_0(2^{*}_{\alpha}-\mu)}(c-I_{\varepsilon}(u))\right)^{1-\frac{2}{2^{*}_{\alpha}}}.
		$$
		Then
		$C_{0}(\eps)\varepsilon^{N-2\alpha}\leq c-I_{\varepsilon}(u)\leq c.$
		If $c<C_{0}(\eps)\varepsilon^{N-2\alpha}$, we get a contradiction, which implies
		$$
		u^{1}_{n_j}\rightarrow 0\quad \textup{in $L^{2^{*}_{\alpha}}(\mathbb{R}^{N},\mathbb{C})$}.
		$$
		It follows from \eqref{16} that
		$$
		\left|\int_{\mathbb{R}^{N}}f(x,|u^{1}_{n_j}|)|u^{1}_{n_j}|^2\,dx\right|
		\leq\int_{\mathbb{R}^{N}}(\lambda|u^{1}_{n_j}|^2+C(\lambda)|u^{1}_{n_j}|^{2^{*}_{\alpha}})\,dx.
		$$
		As $\{u^{1}_{n_j}\}_{j\in {\mathbb N}}$ is bounded in $L^{2}(\mathbb{R}^{N})$, we have
		$$\limsup_{j\rightarrow\infty}\int_{\mathbb{R}^{N}}f(x,|u^{1}_{n_j}|)|u^{1}_{n_j}|^2\,dx
		=\limsup_{\lambda\rightarrow0}\limsup_{j\rightarrow\infty}\int_{\mathbb{R}^{N}}f(x,|u^{1}_{n_j}|)|u^{1}_{n_j}|^2\,dx=0.$$
		By \eqref{17}, it follows that $u^{1}_{n_j}\rightarrow 0$ in $X_\eps$ as $j\to\infty$.
	\end{proof}
	
	\noindent
	Next we provide a result to show that $I_\varepsilon$ has a Mountain Pass geometry.
	
	\begin{lemma}[Mountain Pass geometry I]
		\label{negat} For any $\varepsilon>0$ and $\delta>0$, there exist $t_0=t_0(\eps,\delta)>0$ and  $\psi_{\varepsilon,\delta}\in X_\eps$ such that $I_{\varepsilon}(t_0\psi_{\varepsilon,\delta})<0$.
	\end{lemma}
	\begin{proof}
		
		We first verify that
		\begin{equation*}
			\inf\left\{\int_{\mathbb{R}^{N}}\frac{|\phi(x)-\phi(y)|^2}{|x-y|^{N+2\alpha}}\,dxdy:\phi\in C_0^\infty(\mathbb{R}^{N})\ \textup{with} \ \|\phi\|_{L^q(\mathbb{R}^{N})}=1\right\}=0.
		\end{equation*}
		Let $\phi\in C_0^\infty(\mathbb{R}^{N})$ with $\|\phi\|_{L^q(\mathbb{R}^{N})}=1$ and ${\rm supp}\,\phi\subset B_{r_0}$, where $r_0>0$. Then we have
		$$\int_{\mathbb{R}^{N}}|\delta^{\frac{N}{q}}\phi(\delta x)|^q\,dx=1$$
		and, as $\delta\rightarrow0$,
		$$
		\int_{\mathbb{R}^{2N}}\frac{|\delta^{\frac{N}{q}}\phi(\delta x)-\delta^{\frac{N}{q}}\phi(\delta y)|^2}{|x-y|^{N+2\alpha}}\,dxdy
		=\delta^{\frac{2N-(N-2\alpha)q}{q}}\int_{\mathbb{R}^{2N}}\frac{|\phi(x)-\phi(y)|^2}{|x-y|^{N+2\alpha}}\,dxdy\rightarrow0.
		$$
		Hence, for any $\delta>0$, there exist $r_{\delta}>0$ and   $\phi_{\delta}\in C_0^{\infty}(\mathbb{R}^{N})$ with $\|\phi_{\delta}\|_{L^q(\mathbb{R}^{N})}=1$  and ${\rm supp}\phi_{\delta}\subset B_{r_\delta}$
		such that
		\begin{equation*}
			\int_{\mathbb{R}^{2N}}\frac{|\phi_{\delta}(x)-\phi_{\delta}(y)|^2}{|x-y|^{N+2\alpha}}\,dxdy\leq C\delta^{\frac{2N-(N-2\alpha)q}{q}}.
		\end{equation*}
		Let $\psi_{\delta}(x):=e^{\im A(0)\cdot x}\phi_{\delta}(x)$ and  $\psi_{\varepsilon,\delta}(x):=\psi_{\delta}(\varepsilon^{-1} x)$.
		By $(f_4)$, for any $t>0$ we get
		\begin{eqnarray*}
			\begin{split}
				I_{\varepsilon}(t\psi_{\varepsilon,\delta})
				\leq&\frac{t^2}{2}[\psi_{\varepsilon,\delta}]_{\alpha,A_{\varepsilon}}^2
				+t^2\frac{\varepsilon^{-2\alpha}}{2}\int_{\mathbb{R}^{N}}V(x)|\psi_{\varepsilon,\delta}(x)|^2\,dx-t^q\frac{c_1}{q}\varepsilon^{-2\alpha}
				\int_{\mathbb{R}^{N}}|\psi_{\varepsilon,\delta}(x)|^{q}\,dx\\
				=&\varepsilon^{N-2\alpha}\bigg\{\frac{t^2}{2}\int_{\mathbb{R}^{2N}}\frac{|\psi_{\delta}(x)-e^{\im(x-y)\cdot A(\frac{\varepsilon x+\varepsilon y}{2})}\psi_{\delta}(y)|^2}{|x-y|^{N+2\alpha}}\,dxdy
				+\frac{t^2}{2}\int_{\mathbb{R}^{N}}V(\varepsilon x)|\psi_{\delta}(x)|^2\,dx\\
				&-t^q\frac{c_1}{q}\int_{\mathbb{R}^{N}}|\psi_{\delta}(x)|^{q}\,dx\bigg\}
				=: \varepsilon^{N-2\alpha}J_{\varepsilon}(t\psi_{\delta}).
			\end{split}
		\end{eqnarray*}
		Now it is easy to see that assumption $q>2$ implies there exists $t_0>0$ such that
		$$I_{\varepsilon}(t_0\psi_{\varepsilon,\delta})\leq \varepsilon^{N-2\alpha}J_{\varepsilon}(t_0\psi_{\delta})<0.$$
		This finishes the proof.
	\end{proof}
	
	\noindent
	Let $\psi_{\delta}(x)=e^{\im A(0)\cdot x}\phi_{\delta}(x)$, where $\phi_\delta$ is
	as in the proof of Lemma~\ref{negat}. Then, we have the following

	\begin{lemma}[Norm estimate]
		\label{norm-est}
	 For any $\delta>0$ there exists $\varepsilon_0=\eps_0(\delta)>0$ such that
	$$
	\int_{\mathbb{R}^{2N}}\frac{|\psi_{\delta}(x)-e^{\im(x-y)\cdot A(\frac{\varepsilon x+\varepsilon y}{2})}\psi_{\delta}(y)|^2}{|x-y|^{N+2\alpha}}\,dxdy\leq C\delta^{\frac{2N-(N-2\alpha)q}{q}}+\frac{1}{1-\alpha}\delta^{2\alpha}+\frac{4}{\alpha}\delta^{2\alpha},
	$$
	for all $0<\eps<\eps_0$, for come constant $C>0$ depending only on $[\phi]_{\alpha,0}$.
\end{lemma}

\begin{proof}
 For any $\delta>0$, we have	
	\begin{align*}
			&\int_{\mathbb{R}^{2N}}\frac{|\psi_{\delta}(x)-e^{\im(x-y)\cdot A(\frac{\varepsilon x+\varepsilon y}{2})}\psi_{\delta}(y)|^2}{|x-y|^{N+2\alpha}}\,dxdy\\
			=&\int_{\mathbb{R}^{2N}}\frac{|e^{\im A(0)\cdot x}\phi_{\delta}(x)-e^{\im(x-y)\cdot A(\frac{\varepsilon x+\varepsilon y}{2})}e^{\im A(0)\cdot y}\phi_{\delta}(y)|^2}{|x-y|^{N+2\alpha}}\,dxdy\\
			\leq&2\int_{\mathbb{R}^{2N}}\frac{|\phi_{\delta}(x)-\phi_{\delta}(y)|^2}{|x-y|^{N+2\alpha}}\,dxdy+2\int_{\mathbb{R}^{2N}}
			\frac{|\phi_{\delta}(y)|^2|e^{\im(x-y)\cdot(A(0)-A(\frac{\varepsilon x+\varepsilon y}{2}))}-1|^2}{|x-y|^{N+2\alpha}}\,dxdy.
	\end{align*}
	Next we will  estimate  the second term in the above inequality.
	Notice that
	\begin{equation}
	\label{seno}
		\left|e^{\im(x-y)\cdot(A(0)-A(\frac{\varepsilon x+\varepsilon y}{2}))}-1\right|^2=4\sin^2\left[\frac{(x-y)\cdot(A(0)-A(\frac{\varepsilon x+\varepsilon y}{2}))}{2}\right].
	\end{equation}
	For any $y\in  B_{r_\delta}$, if $|x-y|\leq\frac{1}{\delta}\|\phi_{\delta}\|_{L^2}^{\frac{1}{\alpha}}$, then $|x|\leq r_\delta+\frac{1}{\delta}\|\phi_{\delta}\|_{L^2}^{\frac{1}{\alpha}}$. Hence, we have
	$$
	\left|\frac{\varepsilon x+\varepsilon y}{2}\right|\leq\frac{\varepsilon}{2}\left(2r_\delta+\frac{1}{\delta}\|\phi_{\delta}\|_{L^2}^{\frac{1}{\alpha}}\right).
	$$
	Since $A:\mathbb{R}^{N}\rightarrow\mathbb{R}^{N}$ is continuous, 
	there exists $\varepsilon_0>0$ such that for any $0<\varepsilon<\varepsilon_0$,
	$$
	\left|A(0)-A\left(\frac{\varepsilon x+\varepsilon y}{2}\right)\right|\leq\delta\|\phi_{\delta}\|_{L^2}^{-\frac{1}{\alpha}},
	\quad\text{for $|y|\leq r_\delta$ and $|x|\leq r_\delta+\frac{1}{\delta}\|\phi_{\delta}\|_{L^2}^{\frac{1}{\alpha}}$}.
	$$
	which implies
	$$
	\left|e^{\im(x-y)\cdot (A(0)-A(\frac{\varepsilon x+\varepsilon y}{2}))}-1\right|^2\leq|x-y|^2\delta^2\|\phi_{\delta}\|_{L^2}^{-\frac{2}{\alpha}}.
	$$
	For all $\delta>0$ and $y\in  B_{r_\delta}$, let us define
	$$
	M_{\delta,y}:=\left\{x\in\mathbb{R}^N:|x-y|\leq \frac{1}{\delta}\|\phi_{\delta}\|_{L^2}^{\frac{1}{\alpha}}\right\}.
	$$
	Then gathering the above facts, for all $0<\eps<\eps_0$, we have
	\begin{equation*}
		\begin{split}
			&\int_{\mathbb{R}^{2N}}\frac{|\phi_{\delta}(y)|^2|e^{\im(x-y)\cdot (A(0)-A(\frac{\varepsilon x+\varepsilon y}{2}))}-1|^2}{|x-y|^{N+2\alpha}}\,dxdy\\
			=&\Big(\int_{ B_{r_\delta}}|\phi_{\delta}(y)|^2\,dy
			\int_{M_{\delta,y}}+\int_{ B_{r_\delta}}|\phi_{\delta}(y)|^2\,dy\int_{\mathbb{R}^{N}\setminus M_{\delta,y}}\Big)\frac{|e^{\im(x-y)(A(0)-A(\frac{\varepsilon x+\varepsilon y}{2}))}-1|^2}{|x-y|^{N+2\alpha}}\,dx\\
			\leq&\int_{ B_{r_\delta}}|\phi_{\delta}(y)|^2\,dy\int_{M_{\delta,y}}
			\frac{|x-y|^2}{|x-y|^{N+2\alpha}}\delta^2\|\phi_{\delta}\|_{L^2}^{-\frac{2}{\alpha}}\,dx+\int_{ B_{r_\delta}}|\phi_{\delta}(y)|^2\,dy\int_{\mathbb{R}^{N}\setminus M_{\delta,y}}
			\frac{4}{|x-y|^{N+2\alpha}}\,dx\\
			\leq&\frac{1}{2-2\alpha}\delta^{2\alpha}+\frac{4}{2\alpha}\delta^{2\alpha}.
		\end{split}
	\end{equation*}
Combining the previous inequalities concludes the proof.
\end{proof}

\noindent
Let $t_0=t_0(\eps,\delta)>0$ and $\psi_{\varepsilon,\delta}$ of Lemma~\ref{negat}. Then, we have the following 
	
	\begin{lemma}[Mountain Pass geometry II]
		For any $\varepsilon>0$ and $\delta>0$, there exist
		$$
		d_{\varepsilon,\delta}>0\quad\,\,\text{and}\quad\,\, 0<\rho_{\varepsilon,\delta}<\|t_0\psi_{\varepsilon,\delta}\|_{\varepsilon},
		$$
		with $I_{\varepsilon}(u)\geq d_{\varepsilon,\delta}$ for $u\in X_\eps$ with $\|u\|_{\varepsilon}=\rho_{\varepsilon,\delta}$ and $I_{\varepsilon}(u)>0$ 
		for any $u\in X_\eps\setminus\{0\}$ with $\|u\|_{\varepsilon}<\rho_{\varepsilon,\delta}$. 
	\end{lemma}
	\begin{proof}
		By $(f_1)$ and $(f_2)$, for any $\tau>0$, there exists $C(\tau)>0$ such that
		\begin{equation*}
			|F(x,t)|\leq\tau t^2+C(\tau)|t|^{2^{*}_{\alpha}}.
		\end{equation*}
		For any $u\in X_\eps$, from Proposition \ref{th2.2}, we derive
			\begin{align*}
				I_{\varepsilon}(u)\geq&\frac{1}{2}\|u\|_{\varepsilon}^2-\tau\varepsilon^{-2\alpha}\|u\|^2_{L^2}
				-C(\tau)\varepsilon^{-2\alpha}\|u\|^{2^{*}_{\alpha}}_{L^{2^{*}_{\alpha}}}
				-\varepsilon^{-2\alpha}\frac{K_1}{2^{*}_{\alpha}}\|u\|^{2^{*}_{\alpha}}_{L^{2^{*}_{\alpha}}}\\
				\geq&\frac{1}{2}\|u\|_{\varepsilon}^2-\tau\varepsilon^{-2\alpha} c^2(\varepsilon)\|u\|_{\varepsilon}^2
				-C(\varepsilon)\|u\|^{2^{*}_{\alpha}}_{\varepsilon},
		\end{align*}
		where $c(\varepsilon)>0$ is the embedding constant of $(X_\eps,\|\cdot\|_\varepsilon)\hookrightarrow L^2(\mathbb{R}^{N},\mathbb{C})$.
		Letting $\tau<\frac{\varepsilon^{2\alpha}}{4c^2(\varepsilon)}$,
		we get
		$$
		I_{\varepsilon}(u)\geq\frac{1}{4}\|u\|_{\varepsilon}^2
		-C(\varepsilon)\|u\|^{2^{*}_{\alpha}}_{\varepsilon}.
		$$
		Then, there exist $d_{\varepsilon,\delta}>0$ and $0<\rho_{\varepsilon,\delta}<\|t_0\psi_{\varepsilon,\delta}\|_{\varepsilon}$ such that $I_{\varepsilon}(u)\geq d_{\varepsilon,\delta}$ for $u\in X_\eps$ with $\|u\|_{\varepsilon}=\rho_{\varepsilon,\delta}$ and $I_{\varepsilon}(u)>0$ for any $u\in X_\eps\setminus\{0\}$ with $\|u\|_{\varepsilon}<\rho_{\varepsilon,\delta}$.
	\end{proof}
	
	\begin{proposition}[Sobolev constant bounds]
		\label{Const-lim}
		There exists ${\mathcal S}_\alpha,{\mathcal S}^\alpha>0$ independent of $\eps$ with
		$$
		{\mathcal S}_\alpha\leq S_\alpha^\eps\leq {\mathcal S}^\alpha,\quad\text{for every $\eps>0$}.
		$$
		In particular, with reference to Lemma~\ref{Th3.1}, the Palais-Smale for $I_\eps$ holds for
	\begin{equation}
	\label{boundind}
		-\infty<c<C_0\varepsilon^{N-2\alpha}, \qquad 
		C_{0}:=\left(\frac{{\mathcal S}_{\alpha}}{c_2+K_1}\right)^{\frac{2^{*}_{\alpha}}{2^{*}_{\alpha}-2}}\frac{(2^{*}_{\alpha}-\mu)K_0}{\mu 2^{*}_{\alpha}}.
	\end{equation}
	\end{proposition}
	\begin{proof}
		By virtue of the pointwise diamagnetic inequality \cite[Remark 3.2]{DS}
		\begin{align*}
		\big||u(x)|-|u(y)|\big|\leq \big|u(x)-e^{{\rm i}(x-y)\cdot A_\eps(\frac{x+y}{2})}u(y)\big|,
		\quad\text{for a.e.\ $x,y\in\mathbb{R}^N$ and all $\eps>0$,}
		\end{align*}
		we have
		\begin{align*}
		S_{\alpha}^\eps=\inf_{u\in D^{\alpha}_{A_{\varepsilon}}(\mathbb{R}^{N})\setminus\{0\}}
		\frac{[u]_{\alpha,_{A_\varepsilon}}^2}{\|u\|_{L^{2_{\alpha}^{*}}}^{2}}\geq \inf_{D^{\alpha}_{A_{\varepsilon}}(\mathbb{R}^{N})\setminus\{0\}}
		\frac{\Big(\displaystyle\int_{\mathbb{R}^{2N}}\frac{||u(x)|-|u(y)||^{2}}{|x-y|^{N+2\alpha}}\,dxdy
			\Big)^{1/2}}{\||u|\|_{L^{2_{\alpha}^{*}}}^{2}}\geq{\mathcal S}_\alpha,
		\end{align*}
		where ${\mathcal S}_\alpha>0$ is the Sobolev constant for the embedding $D^\alpha({\mathbb R}^N)\hookrightarrow L^{2_\alpha^*}(\mathbb{R}^N)$.
		Concerning the opposite inequality, fix $\varphi\in C^\infty_c({\mathbb R}^N)\setminus\{0\}$ with 
		$\|\varphi\|_{L^{2_{\alpha}^*}}=1$ and use the function
		$$
		x\mapsto \varphi\left(\frac{x}{\eps}\right)e^{\im A(0)\cdot\frac{x}{\eps}},
		$$
		in the definition of $S_{\alpha}^\eps$.  We have
			\begin{equation*}
			S_{\alpha}^\eps \leq 
			\int_{\mathbb{R}^{2N}}\frac{|\varphi(x)-e^{\im(x-y)\cdot (A(\frac{\varepsilon x+\varepsilon y}{2})-A(0))}\varphi(y)|^2}{|x-y|^{N+2\alpha}}\,dxdy
			\leq {\mathbb I}_1+{\mathbb I}_2, 
				\end{equation*}
				where
					\begin{equation*}
			{\mathbb I}_1=2\int_{\mathbb{R}^{2N}}\frac{|\varphi(x)-\varphi(y)|^2}{|x-y|^{N+2\alpha}}\,dxdy,\quad
			{\mathbb I}_2=2\int_{\mathbb{R}^{2N}}
			\frac{|\varphi(y)|^2|e^{\im(x-y)\cdot(A(0)-A(\frac{\varepsilon x+\varepsilon y}{2}))}-1|^2}{|x-y|^{N+2\alpha}}\,dxdy.
			\end{equation*}
			It is sufficient to estimate ${\mathbb I}_2$ from above independently of $\eps>0$. If $K$ is the support of $\varphi,$ let
			$$
			M_y:=\big\{x\in {\mathbb R}^N: |x-y|\leq 1\big\},\quad y\in K,
			$$
			Taking into account \eqref{seno}, for some $C>0$ independent of $\eps$, we have
			\begin{equation*}
			\begin{split}
			{\mathbb I}_2=&\Big(\int_{K}|\varphi(y)|^2\,dy
			\int_{M_y}+\int_{K}|\varphi(y)|^2\,dy\int_{\mathbb{R}^{N}\setminus M_y}\Big)\frac{|e^{\im(x-y)(A(0)-A(\frac{\varepsilon x+\varepsilon y}{2}))}-1|^2}{|x-y|^{N+2\alpha}}\,dx\\
			\leq& C\int_{K}|\varphi(y)|^2\,dy\int_{M_y}
			\frac{|x-y|^2}{|x-y|^{N+2\alpha}}\,dx
			+C\int_{K}|\varphi(y)|^2\,dy\int_{\mathbb{R}^{N}\setminus M_y}
			\frac{1}{|x-y|^{N+2\alpha}}\,dx=:{\mathcal S}^\alpha>0,
			\end{split}
			\end{equation*}
			concluding the proof.
	\end{proof}
	
	\section{Proof of Theorem~\ref{Th1} concluded}
	We shall prove that there exists $\varepsilon_0>0$ such that 
	for any $\varepsilon\in(0,\varepsilon_0)$, problem \eqref{p} 
	admits a solution $u_\varepsilon\in X_\eps$ close to the trivial one in $X_\eps$ for the norm $\|\cdot\|_{X_\eps}$.
	For any $t>0$, from Lemma~\ref{norm-est} we have
	\begin{align*}
			I_{\varepsilon}(t\psi_{\eps,\delta}) 
			&\leq c_1^{-\frac{2}{q-2}}\frac{q-2}{2q}\varepsilon^{N-2\alpha}
			\Big(\int_{\mathbb{R}^{2N}}\frac{|\psi_{\delta}(x)-e^{\im(x-y)\cdot A(\frac{\varepsilon x+\varepsilon y}{2})}\psi_{\delta}(y)|^2}{|x-y|^{N+2\alpha}}\,dxdy
			+\int_{\mathbb{R}^{N}}V(\varepsilon x)|\psi_{\delta}|^2\,dx\Big)^{\frac{q}{q-2}}\\
			&\leq c_1^{-\frac{2}{q-2}}\frac{q-2}{2q}\varepsilon^{N-2\alpha}\Big(C\delta^{\frac{2N-(N-2\alpha)q}{q}}
			+\frac{1}{1-\alpha}\delta^{2\alpha}+\frac{4}{\alpha}\delta^{2\alpha}
			+\int_{\mathbb{R}^{N}}V(\varepsilon x)|\psi_{\delta}|^2\,dx\Big)^{\frac{q}{q-2}}.
	\end{align*}
Choose now $\delta>0$, depending {\em only} upon $N,\alpha,f,A,K$, such that
$$
c_1^{-\frac{2}{q-2}}\frac{q-2}{2q}\left(C\delta^{\frac{2N-(N-2\alpha)q}{q}}+\frac{1}{1-\alpha}\delta^{2\alpha}+\frac{4}{\alpha}\delta^{2\alpha}
+\delta\right)^{\frac{q}{q-2}}<C_0,
$$
where $C_0$ is defined in \eqref{boundind}. Since  $V(x)\to 0$ as $|x|\to 0$, there is $x_{0,\delta}>0$ with 
$$
|V(x)|<\frac{\delta}{\|\psi_{\delta}\|^2_{L^2}},\quad\text{for all $|x|<x_{0,\delta}$}.
$$
We take $\varepsilon_{1}=\min\Big\{\varepsilon_0, \frac{x_{0,\delta}}{r_{\delta}}\Big\}$. Then, for any $\varepsilon<\varepsilon_1$, we have
	$$
	\int_{B_{r_\delta}}V(\varepsilon x)|\psi_{\delta}(x)|^2\,dx<\delta.
	$$
	From the above estimate, we obtain
	\begin{equation*}
		\max_{t\geq 0}	I_{\varepsilon}(t\psi_{\eps,\delta})<C_0\varepsilon^{N-2\alpha}
	\end{equation*}
	Denote, for every $\eps>0$,
	$$
	c_{\varepsilon}:=\inf\limits_{\gamma\in\Gamma}\max\limits_{t\in[0,1]}I_{\varepsilon}(\gamma(t)),
	\quad\,\, 
	\Gamma_\eps:=\Big\{\gamma\in C([0,1],X_\eps):\gamma(0)=0,\gamma(1)=t_0\psi_{\varepsilon,\delta}\Big\}.
	$$
Then, we have 	
	$$
	\inf\limits_{\|u\|_{\eps}=\rho_{\varepsilon,\delta}}I_{\varepsilon}(u)>I_{\varepsilon}(0)>I_{\varepsilon}(t_0\psi_{\varepsilon,\delta})
	$$
	and, by using the curve $\gamma(t)(x):=tt_0\psi_{\eps,\delta}(x)$ of $\Gamma_\eps$, we get
	\begin{equation}\label{20}
		0<d_{\varepsilon,\delta}\leq c_{\varepsilon}\leq \max_{t\in [0,1]}I_{\varepsilon}(tt_0 \psi_{\varepsilon,\delta})
		\leq\max_{t\geq 0}	I_{\varepsilon}(t\psi_{\eps,\delta})
		<C_0\varepsilon^{N-2\alpha}.
	\end{equation}
	By the Mountain Pass Theorem, there exists a sequence $\{u_{n}\}_{n\in {\mathbb N}}\subset X_\eps$  such that 
	$$
	I_{\varepsilon}(u_{n})\rightarrow c_{\varepsilon}\quad\mbox{and}\quad I_{\varepsilon}'(u_{n})\rightarrow 0\ \mbox{in}\ X^*_\eps,\ \mbox{as}\ n\rightarrow\infty.
	$$
	By Proposition~\ref{Const-lim}, there is a subsequence $\{u_{n_j}\}_{j\in{\mathbb N}}$ such that $u_{n_j}\rightarrow u_\varepsilon$ in $X_\eps$. Thus $I_\varepsilon(u_\varepsilon)=c_\varepsilon$ and $I_\varepsilon'(u_\varepsilon)=0$, namely $u_\varepsilon$ is a nontrivial weak solution of \eqref{p}.
	Besides, from \eqref{20} we get 
	\begin{eqnarray*}
		\begin{split}
			C_0\varepsilon^{N-2\alpha}>c_\varepsilon=&I_{\varepsilon}(u_\varepsilon)-\frac{1}{\mu}\langle I'_{\varepsilon}(u_\varepsilon),u_\varepsilon\rangle\\
			\geq&\left(\frac{1}{2}
			-\frac{1}{\mu}\right)[u_\varepsilon]^2_{\alpha,A_\varepsilon}
			+\left(\frac{1}{2}
			-\frac{1}{\mu}\right)\varepsilon^{-2\alpha}\int_{\mathbb{R}^{N}}V(x)|u_\varepsilon|^2\,dx,
		\end{split}
	\end{eqnarray*}
	which implies that
	$$
	[u_\varepsilon]^2_{\alpha,A_\varepsilon}<\frac{2C_0\mu}{\mu-2}\varepsilon^{N-2\alpha},\qquad
	\int_{\mathbb{R}^{N}}V(x)|u_\varepsilon|^2\,dx<\frac{2C_0\mu}{\mu-2}\varepsilon^{N}.
	$$
	Then $u_\varepsilon\rightarrow0$ in $X_\eps$ for the norm $\|\cdot\|_{X_\eps}$, as $\varepsilon\rightarrow0$.
	\qedsymbol

	\section{Some results without magnetic field}
	In this Section, we consider the existence of solutions for \eqref{p} without  magnetic field, i.e.\ $A\equiv 0$.
	We first establish the existence of $m$ pairs of solutions of via the Ljusternik-Schnirelmann theory of critical points.
	Let $\Sigma(X_\eps)$ be the family of sets $F\subseteq \Sigma(X_\eps)\setminus\{0\}$ such that $F$ is closed in $X_\eps$ and symmetric with respect to $0$, i.e.\ $x\in F$ implies $-x\in F$. For $F\in\Sigma(X_\eps)$, we define the genus of $F$ to be $k$, denoted by ${\rm gen}(F)=k$, if there is a continuous and odd map $\psi: F\rightarrow \mathbb{R}^{k}\setminus\{0\}$ and $k$ is the smallest integer with this property. The definition of  genus  here, which was by {\em Coffman} \cite{Coffman}, is equivalent with the {\em Krasnoselski} original genus.
	Denote by $\Gamma_{*}$ the set of all odd homeomorphisms $g\in C(X_\eps,X_\eps)$ such that $g(0)=0$ and $g(B_1)\subseteq\{u\in X_\eps:I_{\varepsilon}(u)\geq0\}$. We denote by $\Gamma_{m}$ the set of all compact subsets $F$ of $X_\eps$ which are symmetric with respect to the origin and satisfies ${\rm gen}(F\cap g(\partial B_1))\geq m$ for any $g\in\Gamma_{*}$. We refer to \cite{Chabrowski} for more details.
	
	\begin{theorem}\label{Th2}
		Assume that hypotheses  $(V_1)$-$(V_2)$,
		$(f_1)$-$(f_4)$  and $(K)$ are fulfilled. If the subcritical nonlinearity $f(x,t)$ is odd in $t$, for any $m\in\mathbb{N}$  there exist $\varepsilon_{m}>0$
		such that for any $\varepsilon\in(0,\varepsilon_{m})$, problem \textup{(1.1)} has at least $m$ pairs of nontrivial weak solutions in $X_\eps$.
	\end{theorem}
	
	\begin{proof} 
		As in Lemma~\ref{negat},
		for any $m\in\mathbb{N}$,  we can take  $\phi_{\delta}^{j}\in C_0^{\infty}(\mathbb{R}^{N})$ such that, for any $j=1\ldots,m$,
		$$
		{\rm supp}\,\phi_{\delta}^j\subset B_{r_{m,\delta}}(x_{j,\delta}),\quad 
		\|\phi_{\delta}^{j}\|_{L^q}=1, \quad
		[\phi_{\delta}^{j}]_{\alpha,0}<C\delta^{\frac{2N-(N-2\alpha)q}{q}},
		$$
		with $B_{r_{m,\delta}}(x_{i,\delta})\cap B_{r_{m,\delta}}(x_{j,\delta})=\emptyset,$ for any $i\neq j$.
		Set $e_{\varepsilon,\delta}^{j}(x)=\phi_{\delta}^{j}(\varepsilon^{-1}x)$. Thus
		$$ 
		\int_{\mathbb{R}^{2N}}\frac{|e_{\varepsilon,\delta}^{j}(x)-e_{\varepsilon,\delta}^{j}(y)|^2}{|x-y|^{N+2\alpha}}\,dxdy<
		C\delta^{\frac{2N-(N-2\alpha)q}{q}}\varepsilon^{N-2\alpha},\quad
		\int_{\mathbb{R}^{N}}|e_{\varepsilon,\delta}^{j}|^q\,dx=\varepsilon^N.
		$$
		Define $m$-dimensional subspace $F^{\varepsilon,\delta}_{m}:={\rm span}\{e_{\varepsilon,\delta}^{j}\}_{j=1,\ldots,m}$.
		For any  $\delta>0$ with 
		$$
		m^\frac{3q-2}{q-2}c_1^{-\frac{2}{q-2}}\frac{q-2}{2q}\left(C\delta^{\frac{2N-(N-2\alpha)q}{q}}+\delta      \right)^{\frac{q}{q-2}}<C_0.
		$$ 
Let now for any $j=1,\ldots,m$ radii $R_{j,\delta}>0$ with $B_{r_{m,\delta}}(x_{j,\delta})\subset B_{R_{j,\delta}}(0)$.
Therefore, since  $V(x)\to 0$ as $|x|\to 0$, there is $x_{j,\delta}>0$ with 
$$
|V(x)|<\frac{\delta}{\|\phi^j_{\delta}\|^2_{L^2}},\quad\text{for all $|x|<x_{j,\delta}$}.
$$
Then, for any $\varepsilon<\frac{x_{j,\delta}}{R_{j,\delta}}$, we have
$$
\int_{B_{r_{m,\delta}(x_j)}}V(\varepsilon x)|\phi^j_{\delta}(x)|^2\,dx\leq \int_{B_{R_{j,\delta}(0)}}V(\varepsilon x)|\phi^j_{\delta}(x)|^2\,dx<\delta.
$$
Then, for any $\varepsilon<\min_{j=1,\ldots,m}\{\frac{x_{j,\delta}}{R_{j,\delta}}\}$ and
		$u\in F^{\varepsilon,\delta}_{m}$ with $u=\sum_{j=1}^{m}t_{j}e_{\varepsilon,\delta}^{j}$, by $(f_4)$ we get
		\begin{align*}
				& I_{\varepsilon}(u)\leq\frac{1}{2}\int_{\mathbb{R}^{2N}}\frac{|u(x)-u(y)|^2}{|x-y|^{N+2\alpha}}\,dxdy
				+\frac{\varepsilon^{-2\alpha}}{2}\int_{\mathbb{R}^{N}}V(x)u^2\,dx
				-\frac{c_1\varepsilon^{-2\alpha}}{q}\int_{\mathbb{R}^{N}}|u|^q\,dx\\
				\leq&\sum_{j=1}^{m}\Big(m^2\frac{t_j^2}{2}\int_{\mathbb{R}^{2N}}\frac{|e_{\varepsilon,\delta}^j(x)-e_{\varepsilon,\delta}^j(y)|^2}{|x-y|^{N+2\alpha}}\,dxdy
				+\varepsilon^{-2\alpha}m^2\frac{t_j^2}{2}\int_{\mathbb{R}^{N}}V(x)|e_{\varepsilon,\delta}^j|^2\,dx-\varepsilon^{-2\alpha}\frac{c_1}{q}t_j^q\int_{\mathbb{R}^{N}}|e_{\varepsilon,\delta}^j|^{q}\,dx\Big)\\
				\leq& \eps^{N-2\alpha}\sum_{j=1}^{m}\Big(\Big[m^2\int_{\mathbb{R}^{2N}}\frac{|\phi_{\delta}^j(x)-\phi^j_{\delta}(y)|^2}{|x-y|^{N+2\alpha}}\,dxdy
				+m^2\int_{\mathbb{R}^{N}}V(\eps x)|\phi_{\delta}^j|^2\,dx\Big]\frac{t_j^2}{2}-c_1\frac{t_j^q}{q}\Big)\\				
				\leq&m^\frac{3q-2}{q-2}c_1^{-\frac{2}{q-2}}\frac{q-2}{2q}\left(C\delta^{\frac{2N-(N-2\alpha)q}{q}}+\delta      \right)^{\frac{q}{q-2}}\varepsilon^{N-2\alpha}<C_0\varepsilon^{N-2\alpha}.
		\end{align*}
		Since ${\rm dim}(F^{\varepsilon,\delta}_{m})<\infty$, $\|\cdot\|_{L^{q}(\mathbb{R}^{N})}$ and $\|\cdot\|_{\varepsilon}$ are equivalent. Then $I_{\varepsilon}(u)\rightarrow-\infty$, as $u\in F^{\varepsilon,\delta}_{m}$ with $\|u\|_{\varepsilon}\rightarrow\infty$. 
		For any $1\leq j\leq m$, let 
		$$
		c_{\varepsilon}^j=\inf_{F\in \Gamma_{m}}\max_{u\in F}I_{\varepsilon}(u),
		$$
		we have 
		$$
		d_\varepsilon\leq c_{\varepsilon}^1\leq c_{\varepsilon}^2\leq\cdots\leq c_{\varepsilon}^m\leq\sup_{u\in F^{\varepsilon,\delta}_{m}}I_{\varepsilon}(u)\leq C_0\varepsilon^{N-2\alpha}.
		$$ 
		From Proposition \ref{Const-lim}, $I_{\varepsilon}$ satisfies ${\rm (PS)}_{c_{\varepsilon}^j}$ condition. Thus, $c_{\varepsilon}^j$ is a critical value of $I_{\varepsilon}$ and $u_{\varepsilon,j}$ is a critical point of $I_{\varepsilon}$ with $I_{\varepsilon}(u_{\varepsilon,j})=c_{\varepsilon}^j$. As $f(x,t)$ is odd in $t$, we derive that $-u_{\varepsilon,j}$ is also a critical point of $I_{\varepsilon}$.
		Then $I_{\varepsilon}$ has at least $m$ pairs  of nontrivial solutions.
	\end{proof}

	\noindent
	Finally, we verify that  problem \eqref{p}  has one pair of sign-changing solutions.
	Let $g:\mathbb{R}^N\rightarrow\mathbb{R}^N$ be an orthogonal involution. Then the action of $g$ on $X$ is defined by 
	$$
	gu(x)=-u(gx),\quad\text{for any $u\in X_\eps$}. 
$$	
	If $V(gx)=V(x)$, $h(gx)=h(x)$ and $f(gx,t)=f(x,t)$, it is easy to verify that $I_{\varepsilon}$ is $g$-invariant, i.e. $I_{\varepsilon}(gu)=I_{\varepsilon}(u)$ and $I_{\varepsilon}'(gu)=gI_{\varepsilon}'(u)$. The subspace of $g$-invariant functions is defined by 
	$$
	X_g=\{u\in X_\eps:gu=u\}.
	$$
	Then the critical points of $\widetilde{I}_{\varepsilon}=I_{\varepsilon}|_{X_g}$ are critical points of $I_{\varepsilon}$.
	Therefore, it suffices to prove the existence of critical points for $\widetilde{I}_{\varepsilon}$ on $X_g$.
	As a consequence, we obtain the following result:
	
	\begin{theorem}\label{Th3}
		Assume that $(V_1)$-$(V_2)$,
		$(f_1)$-$(f_4)$ and $(K)$ are satisfied. If  the nonlinearity $f(x,t)$ is odd in $t$ and there is an orthogonal involution $g$ such that $V(gx)=V(x)$, $h(gx)=h(x)$ and $f(gx,t)=f(x,t)$, then there exist $\varepsilon^{*}>0$
		such that for any $\varepsilon\in(0,\varepsilon^{*})$, problem \textup{(1.1)} has at least one pair of sign-changing  weak solutions in $X$.
	\end{theorem}
	
	\begin{proof}  Note that for any $\phi\in C_0^\infty(\mathbb{R}^{N})$, 
		$\widetilde{\phi}=\displaystyle\frac{\phi+g\phi}{2}\in C_0^\infty(\mathbb{R}^{N})\cap X_g.$
		One could verify that
		\begin{equation*}
			\inf \left\{\int_{\mathbb{R}^{2N}}\frac{|\phi(x)-\phi(y)|^2}{|x-y|^{N+2\alpha}}\,dxdy:\phi\in C_0^\infty(\mathbb{R}^{N})\cap X_g\ \textup{with} \ \|\phi\|_{L^q(\mathbb{R}^{N})}=1 \right\}=0.
		\end{equation*}
		Then, it is readily seen that $\widetilde{I}_{\varepsilon}$ has a Mountain Pass geometry: for any $\varepsilon>0$ and $\delta>0$:
		
		(1)  there exists $\widetilde{t}_0>0$ and  $\widetilde{e}_{\varepsilon,\delta}\in X_g$ such that $\widetilde{I}_{\varepsilon}(\widetilde{t}_0\widetilde{e}_{\varepsilon,\delta})<0$.
		
		(2)  there exists $\widetilde{d}_\varepsilon>0$ and $0<\widetilde{\rho}_{\varepsilon}<\|t_0\widetilde{e}_{\varepsilon,\delta}\|_{\varepsilon}$ such that $\widetilde{I}_{\varepsilon}(u)\geq \widetilde{d}_\varepsilon$ for any $u\in X_g$ with $\|u\|_{\varepsilon}=\widetilde{\rho}_{\varepsilon}$ and $\widetilde{I}_{\varepsilon}(u)>0$ for any $u\in X_g$ with $\|u\|_{\varepsilon}<\widetilde{\rho}_{\varepsilon}$.
		Denote $$\widetilde{c}_{\varepsilon}=\inf\limits_{\gamma\in\Gamma}\max\limits_{t\in[0,1]}\widetilde{I}_{\varepsilon}(\gamma(t)),$$ where $\Gamma=\{\gamma\in C([0,1],X_g):\gamma(0)=0,\gamma(1)=\widetilde{t}_0\widetilde{e}_{\varepsilon,\delta}\}$.  Then, there is 
		$\varepsilon^*>0$ with, for $0<\varepsilon<\varepsilon^*$,
		\begin{align*}
		\inf\limits_{\|u\|_{\varepsilon}=\widetilde{\rho}_{\varepsilon}}\widetilde{I}_{\varepsilon}(u)&>\widetilde{I}_{\varepsilon}(0)>\widetilde{I}_{\varepsilon}(\widetilde{t}_0\widetilde{e}_{\varepsilon,\delta}), \\
			0<\widetilde{d}_\varepsilon\leq \widetilde{c}_{\varepsilon}\leq \widetilde{I}_{\varepsilon}(t\widetilde{t}_0 \widetilde{e}_{\varepsilon,\delta})&\leq 
		c_1^{-\frac{2}{q-2}}\frac{q-2}{2q}\big(C\delta^{\frac{2N-(N-2\alpha)q}{q}}+\delta  \big)^{\frac{q}{q-2}}\eps^{N-2\alpha}
			<C_0\varepsilon^{N-2\alpha}.
		\end{align*}
		where $C_0$ is as in Proposition \ref{Const-lim}. Then there exists $\widetilde{u}_{\varepsilon}\in X_g$ such that $\widetilde{I}_{\varepsilon}'(\widetilde{u}_{\varepsilon})=0$. Then, $\widetilde{u}_{\varepsilon}$ is a critical point of $I_{\varepsilon}$ and $\widetilde{u}_{\varepsilon}(x)=g\widetilde{u}_{\varepsilon}(x)=-\widetilde{u}_{\varepsilon}(gx)$. It is easy to show that $\widetilde{u}_{\varepsilon}(gx)$  is also a critical point of $I_{\varepsilon}$ and $\widetilde{u}_{\varepsilon}(x)$, $\widetilde{u}_{\varepsilon}(gx)$ change sign.  \qedsymbol
	\end{proof}

\end{document}